\documentclass[a4paper,10pt]{article}

\usepackage{mathtools}
\usepackage{amsfonts}
\usepackage{amssymb}
\usepackage{amsthm}
\usepackage{lmodern}
\usepackage{graphicx}
\usepackage{subcaption}
\usepackage{tikz}
\usepackage{tabularx}
\usepackage[margin=3cm]{geometry}
\usepackage[english]{babel}
\usepackage[T1]{fontenc}
\usepackage[utf8]{inputenc}
\usepackage{stmaryrd} 
 
\usepackage[notcite, notref]{showkeys}
\usepackage{hyperref}
\usepackage{enumitem}

\newtheorem{definition}{Definition}
\newtheorem{theorem}{Theorem}[section]
\newtheorem{lemma}[theorem]{Lemma}

\newtheorem{remark}[theorem]{Remark}
\newtheorem{corollary}[theorem]{Corollary}

\newtheorem{assumption}{Assumption}[]

\newcommand{\dd}{\mathrm{d}}
\newcommand{\N}{\mathbb{N}}
\newcommand{\R}{\mathbb{R}}
\newcommand{\Co}{\mathcal{C}}
\newcommand{\M}{\mathcal{M}}
\newcommand{\Ker}{\text{Ker}}
\newcommand{\Rg}{\mathcal{R}}
\renewcommand{\deg}{\text{deg}}
\newcommand{\dist}{\textnormal{dist}}
\renewcommand{\L}{\mathcal{L}}
\newcommand{\supp}{\mathrm{supp}\,}
\newcommand{\ep}{\mathcal{\varepsilon}}

\newtheorem*{ex*}{Exercice}

	\newcounter{stepnum}

\makeatletter{}

\makeatother{}

\allowdisplaybreaks

\title{Concentration estimates in a multi-host epidemiological model structured by phenotypic traits}
\date{}
\author{Jean-Baptiste Burie$^{(a)}$, Arnaud Ducrot$^{(b)}$, Quentin Griette$^{(a)}$ and Quentin Richard$^{(c)}$ \\
$^{(a)}$ IMB, UMR CNRS 5251, Université de Bordeaux. \\
F-33400 Talence, France. \\
$^{(b)}$ Normandie Univ, UNIHAVRE, LMAH, FR-CNRS-3335,\\ ISCN, 76600 Le Havre, France.\\
$^{(c)}$ MIVEGEC, IRD, UMR CNRS 5290, Université de Montpellier. \\
F-34394 Montpellier, France.}
\begin{document}
\maketitle

\begin{abstract}
In this work we consider a nonlocal system modelling the evolutionary adaptation of a pathogen within a multi-host population of plants.
	Here we focus our analysis on the study of the stationary states. We first discuss the existence of nontrivial equilibria using dynamical system arguments. Then we introduce a small parameter $0<\varepsilon\ll 1$ that characterises the width of the mutation kernel, and we describe the asymptotic shape of steady states with respect to $\varepsilon$. In particular, for $\ep\to 0$ we show that the distribution of the pathogen approaches a singular measure concentrated on the maxima of fitness in each plant population.  This asymptotic description allows us to show the local stability of each of the positive steady states for $\varepsilon \ll 1$, from which we deduce a  uniqueness result for the nontrivial stationary states by means of a topological degree argument. These analyses rely on a careful investigation of the spectral properties of some {nonlocal} operators.
\end{abstract}

\vspace{0.2in}\noindent \textbf{Keywords:} Nonlocal equation, steady state solutions,  concentration phenomenon, epidemiology, population genetics.

\vspace{0.1in}\noindent \textbf{2010 Mathematical Subject Classification:} 35B40, 35R09, 47H11, 92D10, 92D30.

\section{Introduction}

In this work we study the stationary states of the following system of equations 
\begin{equation}\label{Eq:Model}
\begin{dcases} 
   \frac{\dd S_k}{\dd t}  (t)=\xi_k \Lambda -\theta S_k(t)-S_k(t)\int_{\R^N}\beta_k( y)A(t,y)\dd y, \quad k=1,2, \\  
  \frac{\partial I_k}{\partial t}  (t,x)= \beta_k(x)S_k(t)A(t,x)-(\theta+d_k(x))I_k(t,x),\quad  k=1,2, \\  
  \frac{\partial A}{\partial t}  (t,x)= -\delta A(t,x)+\int_{\R^N}m_\ep (x-y)\left[r_1(y)I_1(t,y)+r_2(y)I_2(t,y)\right]\dd y.
\end{dcases}
\end{equation} 
The above system describes the evolution of a pathogen producing spores in a heterogeneous plant population with two hosts. This model has been proposed in \cite{Iacono2012} to study the impact of resistant plants on the evolutionary adaptation of a fungal pathogen.

Here the state variables are nonnegative functions. The function $S_k(t)$ denotes the healthy tissue density of each host $k\in\{1,2\}$, $I_k(t,x)$ represents the density of tissue infected by pathogen with phenotypic trait value $x\in\R^N$, while  $A(t,x)$ describes the density of airborne spores of pathogens with phenotypic trait value $x\in\R^N$. Here $N\in \mathbb N\setminus\{0\}$ is a given and fixed integer.

The positive parameters $\Lambda, \theta, \delta$ respectively denote  the influx of total new healthy tissue, the death rate of host tissue and the death rate of the spores. The parameters $\xi_k \in(0,1)$ correspond to the proportions of influx of new healthy tissue for each host population and therefore satisfy the relation $\xi_1+\xi_2=1$. Note that in the absence of the disease, namely when $I_1=I_2=A=0$, the density of tissue at equilibrium for each host $k$ is equal to $\xi_k\Lambda/\theta$.

  The phenotypic traits of the pathogen considered in the model are supposed to influence the functions $r_k$, $\beta_k$  and $d_k$ that respectively denote the {spores'} production rates, the infection efficiencies and the infectious periods of the pathogen. Those parameters depend on the phenotypic value $x\in\R^N$ and the host $k=1,2$. 

The function $m_\ep$ is a probability kernel that characterises the mutations arising during the reproduction process. More precisely, given  tissue infected by a mother spore with phenotypic value $y$, $m_\ep(x-y)$ stands for the probability that a produced spore has a phenotypic value $x$. Therefore $m_\ep$ describes the dispersion in the phenotypic trait space $\R^N$ arising at each production of new spores. 
 
Here  we consider that produced spores cannot have a very different phenotypic value from the one of their mother. In other words, mutations are occurring within a small variance so that we assume that the mutation kernel is highly concentrated and depends on a small parameter $0<\ep\ll 1$ according to the following scaling form 
\begin{equation*}
m_{\ep }(x)=\dfrac{1}{\ep ^N}m\left(\dfrac{x}{\ep }\right), \quad \forall x\in \R^N,
\end{equation*}
where $m$ is a  fixed probability distribution (see Assumption \ref{Assump:Psi} in Section \ref{Sec:Main} below).\\

In this work we aim at studying the existence and uniqueness of nontrivial steady states for the above system of equations. We also investigate the shape of these steady states for $\ep\ll 1$ and we shall more precisely study their concentrations  around some specific phenotypic trait values when the mutation kernel is very narrow, \textit{i.e.} for $\ep\to 0$.\\

The above problem supplemented with an age of infection structure has been investigated by Djidjou \textit{et al.} \cite{Djidjou2018} using formal asymptotic expansions and numerical simulations. In the aforementioned work, the authors proved the convergence of the solution of the Cauchy problem toward highly concentrated steady states {(see \cite[Section 4]{Djidjou2018})}.

Moreover, the case of a single host population has already been studied thoroughly. A refined mathematical analysis of the stationary states has been carried out in \cite{Djidjou2017} with a particular emphasis on the concentration property for $\ep\ll 1$. We also refer to \cite{Burie2018, Burie2018b} for the study of the dynamical behaviour and the transient regimes of a corresponding simplified Cauchy problem for a single host.

 Model \eqref{Eq:Model} is related to the selection-mutation models for a population structured by a continuous phenotypic trait introduced in \cite{CrowKimura1964,Kimura} to study the maintenance of genetic variance in quantitative characters. Since then, several studies have been devoted to this class of models in which mutation is frequently  modelled  by either a nonlocal or a Laplace operator.  
In many of these works the existence of steady state solutions  is related to the existence of a positive eigenfunction of some linear operator and to the Krein-Rutman Theorem, see \textit{e.g.}  \cite{Bonnefon2017,Burger1996,Calsina2005,Calsina2013}. In particular, in \cite{Calsina2005,Calsina2013} it is assumed that the rate of mutations is small; in this case the authors are able to prove that the steady state solutions tend to concentrate around some specific trait in the phenotypic space as the mutation rate tends to 0. In \cite{Alfaro2018}, the steady state solutions for a nonlocal reaction-diffusion model for adaptation are given in terms of the principal eigenfunction of a  Schr\"odinger operator. 

As far as dynamical properties are concerned and under the assumption of small mutations, another fruitful approach introduced in \cite{Diekmann2005} consists in proving that the  solutions of the mutation selection problem are asymptotically given by a Hamilton-Jacobi equation. This approach has led to many works, see \textit{e.g.} \cite{Mirra2011,Mirra2012,Taing2018}.

Propagation properties have also been investigated in related models, see {\it e.g.} \cite{ABD2019} and \cite{Griette2019} for spatially  distributed systems of equations. \\

As already mentioned above, in this paper we are concerned with the steady states of \eqref{Eq:Model}.
Using the symbol  $\star$ to denote the convolution product in $\R^N$, steady state solutions of \eqref{Eq:Model} solve the following system of equations
\begin{equation}\label{Stationary}
\begin{dcases} 
   S_k=\frac{\xi_k \Lambda}{\theta+\int_{\R^N}\beta_k( y)A(y)\dd y}, \quad k=1,2, \\  
   I_k(x)= \frac{\beta_k(x)}{\theta+d_k(x)}S_k A(x),\quad  k=1,2, \\  
   \delta A(x)=m_\ep \star \left[r_1(\cdot) I_1(\cdot)+r_2(\cdot)I_2(\cdot)\right](x).
\end{dcases}
\end{equation}
The above system can be rewritten  in the form of a single  equation for $A=A^\ep\in L_+^1(\R^N)$ 
\begin{equation}\label{Eq:AepsIntro}
A^\ep=T^\ep\left(A^\ep\right),
\end{equation}
where the nonlinear operator $T^\ep$ is given by
\begin{equation}\label{T}
T^\ep(\varphi)=\sum_{k=1,2} \frac{ L_k^\ep(\varphi)}{1+\theta^{-1}\int_{\R^N}\beta_k(z)\varphi(z)\dd z}.
\end{equation}
Here, for $k=1,2$, $L_k^\ep$ denotes the following linear operator
\begin{equation}\label{Eq:Lk_ep}
L_k^\ep = \frac{\Lambda \xi_k}{\theta} m_\ep\star \left(\Psi_k\cdot\right),\qquad k=1,2,
\end{equation}
wherein $\Psi_k$ corresponds to the fitness function of the pathogen in  host $k$ 
\begin{equation}\label{Eq:Psi_k} 
\Psi_k(x)=\dfrac{\beta_k(x)r_k(x)}{\delta(\theta+d_k(x))}, \qquad  x\in \R^N, \qquad  k=1,2.
\end{equation}
Conversely, if $A^\ep\in L^1_+(\R^N)$ is a fixed point of $T^\ep$, a stationary solution $\left(S_1^\ep,S_2^\ep,I_1^\ep,I_2^\ep,A^\ep\right)$ to the original system \eqref{Eq:Model} can be reconstructed by injecting $A^\ep$ into the first two equations of    \eqref{Stationary}. The trivial solution $A^\ep\equiv 0$ is always solution of \eqref{Stationary} and corresponds to the disease-free equilibrium.  When  $A^\ep$ is nontrivial,  the corresponding stationary state is said to be endemic.

This paper is organized as follows. In Section \ref{Sec:Main}, we state the main results obtained in this work. In Section \ref{Sec:Thm1} we prove the existence of an  endemic (nontrivial) equilibrium for model \eqref{Eq:Model} by using dynamical system arguments and the theory of global attractors. In Section \ref{Sec:Thm2} we prove  that any nontrivial fixed point of \eqref{T}  
roughly behaves as the superposition of the solution of two single host problems, corresponding to the fixed points of the non-linear operators
\begin{equation}\label{eq:Tk}
T_k^\ep[\varphi]=	\frac{\xi_k\Lambda}{\theta} \frac{m_\ep\star \left(\Psi_k \varphi\right)}
{1+\theta^{-1}\int_{\mathbb R^N}\beta_k(z)\varphi(z) \dd z},
\end{equation}
provided the fitness functions $\Psi_k$ defined in \eqref{Eq:Psi_k} have disjoint supports.
 Finally, in Section \ref{Sec:Thm3}, we investigate the uniqueness of the non-trivial fixed point of $T^\ep$, for $\ep\ll 1$. Our analysis relies on the precise description of the shape of $A^\ep$ coupled with topological degree theory.

\section{Main results and comments} \label{Sec:Main}

In this section we state and discuss the main results that are proved in this paper. Throughout this manuscript we make the following assumption on the model parameters.
\begin{assumption}\label{Assump:Psi}
We assume that 
	\begin{enumerate}[label={\bf\alph*)}]
\item the parameters $\xi_1$, $\xi_2$, $\Lambda$, $\theta$ and $\delta$ are positive constants with $\xi_1+\xi_2=1$;
\item for each $k=1,2$, the functions $\beta_k$, $d_k$,  $r_k$ are continuous, nonnegative and bounded on $\R^N$ and the function $\Psi_k$ defined in \eqref{Eq:Psi_k} is not identically $0$ and  satisfies
$$
 \lim_{\|x\|\to \infty}\Psi_k(x)=0;
$$  
\item the function $m\in L^\infty_+(\R^N)\cap L^1_+(\R^N)$ is positive almost everywhere, symmetric and with  unit mass, \textit{i.e.}
	\begin{equation}\label{eq:supinL1}
		m(x) >0, \;m(-x)=m(x) \; \text{a.e. in } \; \R^N, \text{ and }\int_{\R^N}m(x)\dd x=1.
	\end{equation}
Moreover for every $R>0$, the function satisfies
\begin{equation*}
x\longmapsto \sup_{\|y\|\leq R}m(x+y)\in L^1(\R^N).
\end{equation*}
\end{enumerate}    
\end{assumption} 

{Assumption \ref{Assump:Psi} {c)} provides a simple condition on $m$ ensuring the compactness of the linear operators $L_k^\varepsilon$, see Lemma \ref{Lemma:Prel_1} for details. Note that such a condition holds true for a large class of kernel functions, and in particular for bounded kernels satisfying the following decay estimate at infinity for some $\alpha>N$ 
		\begin{equation*}
			m(x)=O\left(\frac{1}{\|x\|^\alpha}\right)\text{ as }\|x\|\to\infty.
		\end{equation*}
}

As already mentioned in the Introduction, in this work we discuss some properties of the nonnegative fixed points for the nonlinear operator $T^\ep$ in $L^1(\R^N)$. Recall that $A\equiv 0$ is always a solution of such an equation. Our first result provides a sharp condition for the existence of a nontrivial fixed point. This condition relies on the spectral radius $r_\sigma(L^\ep)$ of the linear bounded  operator  $L^\ep\in \mathcal L\left(L^1(\R^N)\right)$ defined by
\begin{equation}\label{Eq:L_ep}
L^\ep(\varphi) := L_1^\ep(\varphi)+L_2^\ep(\varphi)= \dfrac{\Lambda}{\theta} m_\ep\star\left[\left(\xi_1\Psi_1+\xi_2\Psi_2\right)\varphi\right], \forall \varphi\in L^1(\R^N).
\end{equation}
 Our first result reads as follows.
\begin{theorem}[Equilibrium points of System \eqref{Eq:Model}]\label{Thm:main}
Let Assumption \ref{Assump:Psi} be satisfied and let $\ep>0$ be given.
\begin{enumerate} 
\item[(i)] If $r_\sigma(L^\ep)\leq 1$, then $A\equiv 0$ is the unique solution of \eqref{Eq:AepsIntro} in $L^1_+(\R^N)$.
\item[(ii)] If $r_\sigma(L^\ep)>1$, then  there exists at least a continuous function $A^\ep>0$ such that
\begin{equation*}
A^\ep\in L^1(\R^N)\cap L^\infty(\R^N)\text{ and }A^\ep=T^\ep\left(A^\ep\right),
\end{equation*}
where the nonlinear operator $T^\ep$ is defined in \eqref{T}.
Furthermore, the solution $A^\ep$ belongs to $\Co_b(\R^N)$, the space of bounded and continuous functions on $\R^N$, and the family $\{A^\ep\}_{\ep>0}$ is uniformly bounded in $L^1(\R^N)$.
\end{enumerate} 
\end{theorem}

The proof of the above Theorem involves the theory of global attractors  applied to  the discrete dynamical system generated by $T^\ep$. Note that the operator $L^\ep$ is the Fréchet derivative of $T^\ep$ (see \eqref{T}) at $A \equiv 0$. The position of the spectral radius $r_\sigma(L^\ep)$ with respect to $1$ describes the stability and instability of the extinction state $A\equiv 0$ for the aforementioned dynamical system.

In our next result we consider  the situation where $r_\sigma(L^\ep)>1$ and investigate the shape of the nontrivial and nonnegative solutions of the fixed point problem \eqref{Eq:AepsIntro} for $\ep\ll 1$. Observe that the threshold $r_\sigma(L^\ep)$ converges to a limit when $\ep\to0$
\begin{equation}\label{R0}
\lim_{\ep\to 0} r_\sigma(L^\ep)=R_0:=\frac{\Lambda}{\theta}\|\xi_1\Psi_1+\xi_2\Psi_2\|_{L^\infty}.
\end{equation}
{A rigorous proof of this limit can be found in Lemma \ref{Lemma:Prel_1} item \ref{item:symmetric-spectrum}.} In addition to Assumption \ref{Assump:Psi},  we  introduce further conditions on the functions $\beta_k$ and on the decay rate of the mutation kernel $m$. 
\begin{assumption}\label{Assump:supports}
We assume that the mutation kernel satisfies, for all $n\in\mathbb N$,  
\begin{equation*}
		\lim_{\|x\|\to \infty} \Vert x\Vert^n m(x)=0.
\end{equation*}
In other words, $m$ satisfies $m(x)=o\left(\frac{1}{\|x\|^\infty}\right)$ as $\|x\|\to \infty$.

Furthermore, we assume that functions $\beta_1$ and $\beta_2$ have compact supports, separated  in the sense 
\begin{equation}\label{support}
{\rm dist}\,\left(\Sigma_1,\Sigma_2\right)>0\text{ with }\Sigma_k=\overline{\{x\in\R^N,\;\beta_k(x)>0\}},\;k=1,2,
\end{equation}
where ${\dist}$ is the usual distance between sets in $\R^N$
	\begin{equation*}
		\dist(\Sigma_1, \Sigma_2):=\inf_{x\in\Sigma_1}\inf_{y\in\Sigma_2}\|y-x\|.
	\end{equation*}
\end{assumption}
This second assumption  will allow us to  reduce the study of the fixed points of $T^\ep$ to the two simpler fixed point problems associated with $T_k^\ep$ (defined in \eqref{eq:Tk}) weakly coupled when $\ep\ll 1$. {From \eqref{support}, we observe that the pathogen can infect the host $k\in\{1,2\}$ only if its phenotypic value belongs to $\Sigma_k$. Consequently, the latter assumption implies that the first host is immunized to the pathogen with phenotypic values within $\Sigma_2$, and conversely.}

Our last assumption concerns the spectral gap of the  bounded linear operators $L_k^\ep$ (see \eqref{Eq:Lk_ep}). 
Let us recall that for each $\ep>0$ and $k=1,2$, the spectrum $\sigma\left(L_k^\ep\right)$ of $L_k^\ep$ is composed of isolated eigenvalues (except 0) with finite algebraic multiplicities, among which $r_\sigma(L_k^\ep)$ is a simple eigenvalue. Moreover,
\begin{equation}\label{Eq:R0,k}
\lim_{\ep\to 0}r_\sigma(L_k^\ep)=R_{0,k}:=\dfrac{\xi_k\Lambda }{\theta}\|\Psi_k\|_{L^\infty},\;k=1,2.
\end{equation}
We refer to Appendix \ref{Appendix-spectral} for a precise statement of those spectral properties. Recalling the definition of $R_0$ in \eqref{R0}, observe that, due to Assumption \ref{Assump:supports}, we have
$$
R_0=\max\{R_{0,1}, R_{0,2}\}.
$$
Next for $k=1,2$ we denote by $\lambda^{\ep,1}_k>\lambda^{\ep,2}_k$ the first and the second eigenvalues of the linear operator $L^\ep_k$ and we assume that the spectral gaps are not too small, namely 
\begin{assumption}[Spectral gap]\label{Assump:gap}	
We assume that for each $k=1,2$  there exists $n_k\in \N$ such that
$$\liminf_{\ep\to 0} \dfrac{\lambda^{\ep,1}_k-\lambda^{\ep,2}_k}{\ep^{n_k}}>0.$$
\end{assumption}  
Note that the above assumption is satisfied for rather general functions $\Psi_k$. An asymptotic expansion of the first eigenvalues of the operators $L_k^\ep$ has been obtained in \cite{Djidjou2017} when the mutation kernel has a fast decay at infinity and when  $\Psi_k$ are smooth functions. In that case, the asymptotic expansions for the first eigenvalues involve the derivative of the fitness functions $\Psi_k$ at their maximum. Roughly speaking, for each $k=1,2$, Assumption \ref{Assump:gap} is satisfied when each -- partial -- fitness function  $\Psi_k$ achieves its global maximum at a finite number of optimal traits, and its behaviour around any two optimal traits differs  by some derivative. Assumption \ref{Assump:gap} allows us to include the situation studied in \cite{Djidjou2017} in a more general framework. A similar abstract assumption has been used in \cite{Burie2018, Burie2018b} to derive refined information on the asymptotic and the transient behaviour of the solutions to \eqref{Eq:Model} in the context of a single host population.   

The single host problem
\begin{equation}\label{uncoupled}
T_k^\ep\left(A_k^\ep\right)=A_k^\ep 
\end{equation}  
has been extensively studied in Djidjou \textit{et al.} \cite{Djidjou2017}. 
In particular it has been shown that,  when $R_{0,k}>1$,  this equation admits a unique positive solution $A^{\ep, *}_k\in L^1_+(\R^N)$ as soon as $\ep$ is sufficiently small. 
Our next result shows that any nontrivial solution of \eqref{Eq:AepsIntro} is close to the superposition of the  solutions to the two uncoupled problems \eqref{uncoupled} for $k=1,2$, when  $\ep\ll 1$. 
\begin{theorem}[Asymptotic shape of the solutions of \eqref{Eq:L_ep}]\label{Thm:shape}
Let Assumptions \ref{Assump:Psi}, \ref{Assump:supports} and \ref{Assump:gap} be satisfied and assume further that $R_0>1$. Let $A^\ep\in L^1_+(\R^N)\cap L^\infty(\R^N)$ be a nontrivial solution of \eqref{Eq:AepsIntro}. Then the following estimate holds for $\ep\ll 1$:
	\begin{equation*}
		\Vert A^\ep-(A^{\ep, *}_1+A^{\ep, *}_2)\Vert_{L^1(\mathbb R^N)}=o(\ep^\infty), 
	\end{equation*}
	where, for $k=1,2$, $A^{\ep, *}_k\in L^1(\mathbb R^N)$ is the unique positive fixed-point of $T^\ep_k$ if $R_{0,k}>1$ and $A^{\ep, *}_k\equiv 0$ otherwise.
\end{theorem}

\begin{remark} As will  be shown in Lemma \ref{Lemma:Decay}, it should be noted that $\Vert A_1^{\varepsilon, *}\Vert_{L^1(\Sigma_2)}=o(\varepsilon^\infty)$ and, similarly, $\Vert A_2^{\varepsilon, *}\Vert_{L^1(\Sigma_1)}=o(\varepsilon^\infty)$. Therefore, the following result holds as well
	\begin{gather*}
		\begin{aligned}
			\Vert A^\ep-A^{\ep, *}_1\Vert_{L^1(\Sigma_1)}&=o(\varepsilon^\infty), & \Vert A^\ep-A^{\ep, *}_2\Vert_{L^1(\Sigma_2)}&=o(\ep^\infty), 
		\end{aligned} \\
		\Vert A^\ep - (A^{\ep, *}_1+A^{\ep, *}_2)\Vert_{L^1(\mathbb R^N\backslash (\Sigma_1\cup\Sigma_2))}=o(\varepsilon^\infty).
	\end{gather*}
\end{remark} 


\begin{figure}[!h]
\begin{center} 
\includegraphics[width=\linewidth]{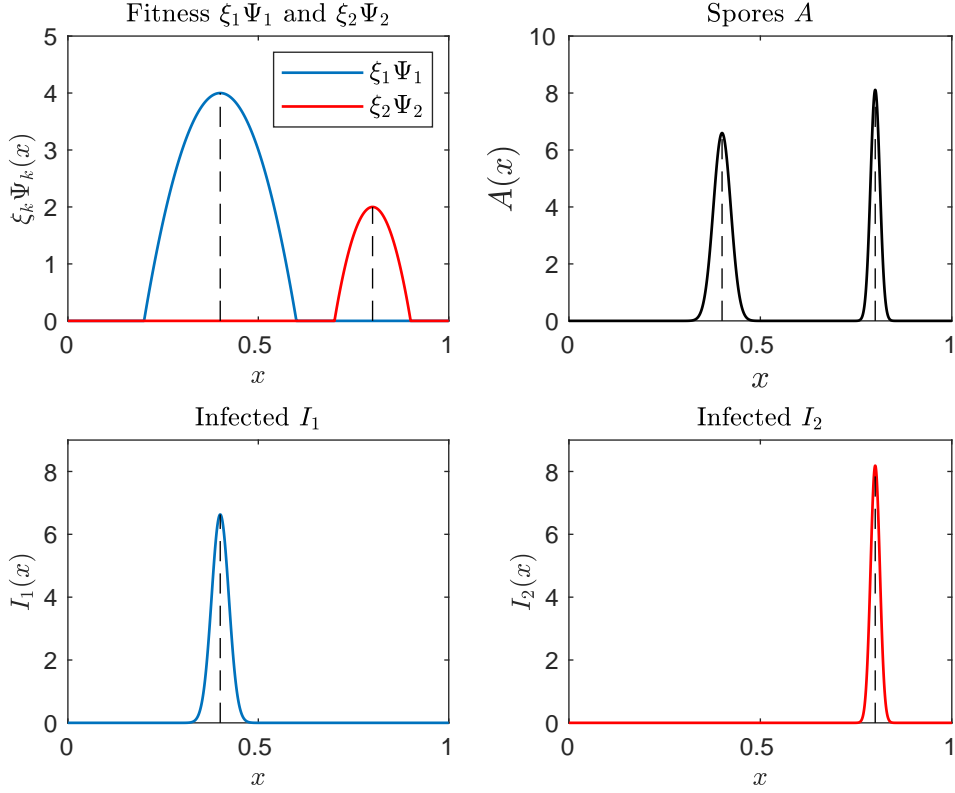}   
\end{center}
\caption{Fitness functions and endemic equilibrium in the case $R_{0,1}>1$ and $R_{0,2}>1$. {Parameters for this simulation are $\xi_1=\xi_2=1/2$, $\Lambda=\theta=\delta=1$, $d_1 \equiv d_2 \equiv 0$, $r_1 \equiv  r_2\equiv 1$, $\beta_1(x)=200((x-0.2)(0.6-x))^+$ and $\beta_2(x)=400((x-0.7)(0.9-x))^+$ where $(\cdot)^+$ denotes the positive part, kernel function is $m(x)=\frac{1}{2}e^ {-|x|}$ and $\varepsilon=10^ {-3}$. } 
}\label{Fig1}
\end{figure}

\begin{figure}[!h]
\begin{center} 
\includegraphics[width=\linewidth]{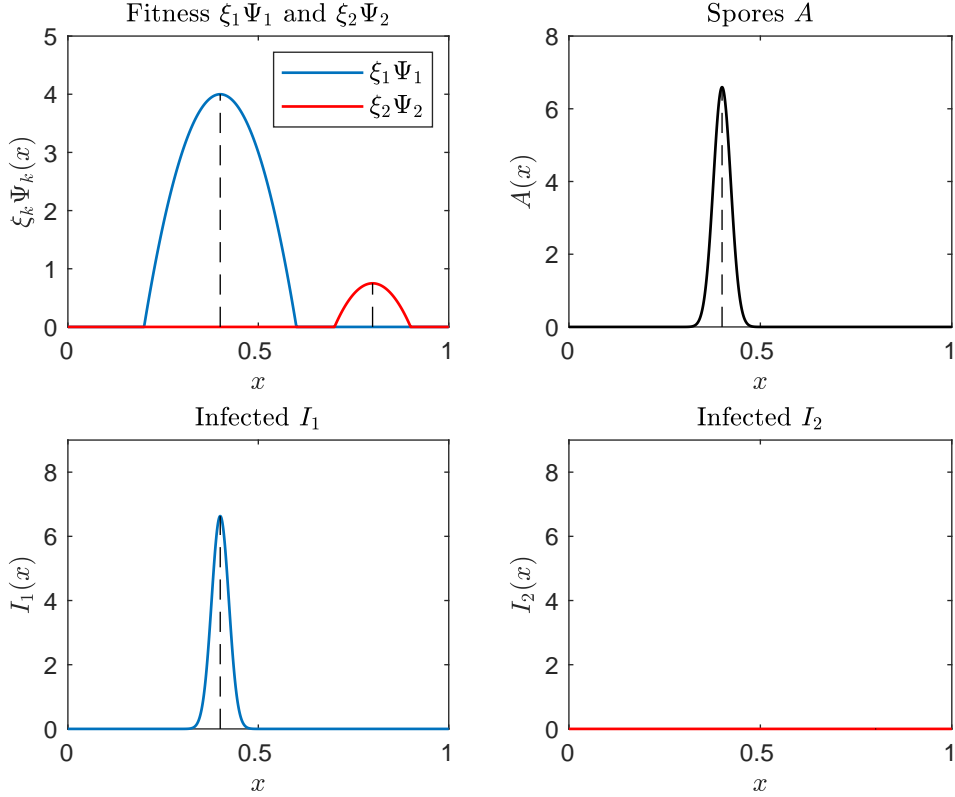}   
\end{center}
\caption{Fitness functions and endemic equilibrium in the case $R_{0,1}>1$ and $R_{0,2}<1$.  {Parameters for this simulation are the same as those of Figure~\ref{Fig1} except for $\beta_2$ which is $\beta_2(x)=150((x-0.7)(0.9-x))^+$ here.} }\label{Fig2}
\end{figure}

In particular, Theorem \ref{Thm:shape} ensures a concentration property for the nontrivial fixed point solutions of \eqref{Eq:AepsIntro} and thus for the endemic solutions of \eqref{Eq:Model} as $\ep \to 0$ (see Figures \ref{Fig1} and \ref{Fig2}). It shows that each infectious population $I_k$ concentrates around phenotypic values maximising $\Psi_k$ if $R_{0,k}>1$ or goes to $0$ a.e. if $R_{0,k}\leq 1$. As a special case, when each $\Psi_k$ achieves its maximum at a single point $x_k\in\Sigma_k$, a slightly more precise result can be stated.

\begin{corollary}[Concentration property of the endemic equilibrium points]\label{Thm:concentration}
Assume that each fitness function $\Psi_k$ admits a unique maximum at $x=x_k$ and that $R_{0,k}>1$ for all $k=1,2$, that is
$$
R_{0,k}=\dfrac{\xi_k\Lambda}{\theta}\Psi_k(x_k)>1,\;\forall k=1,2.
$$
For $\ep\ll 1$, denote by $(S_1^\ep,S_2^\ep, I_1^\ep, I_2^\ep,A^\ep)$ any endemic equilibrium point of \eqref{Eq:Model}. Then,  as $\ep\to0$, the following behaviour holds
$$
	\lim_{\ep\to 0} S_k^\ep=\dfrac{1}{\Psi_k(x_k)} 
$$
and for any function $f$ continuous and bounded on $\R^N$, we have
$$
	\lim_{\ep\to 0} \int_{\R^N} f(x)I_k^\ep(x) \dd x=\dfrac{R_{0,k}-1}{\Psi_k(x_k)\left(1+\frac{d_k(x_k)}{\theta}\right)}  f(x_k)
$$
and
$$
	\lim_{\ep\to 0} \int_{\R^N} f(x)A^\ep(x) \dd x= \frac{\theta}{\beta_1(x_1)}(R_{0,1}-1)f(x_1) + \frac{\theta}{\beta_2(x_2)}(R_{0,2}-1)f(x_2).
$$
\end{corollary}

Numerical explorations suggest that the latter concentration property may fail to hold when Assumption \ref{Assump:supports} does not hold. Indeed, we can find examples where $R_{0,1}>1, R_{0,2}>1$ and where the population of spores does not concentrate to either maximum of $\Psi_1$ or $\Psi_2$. Such an example is shown in  Figure \ref{Fig3}.
{
More precisely, the concentration of the steady states as $\varepsilon$ goes to $0$ in absence of Assumption \ref{Assump:supports} is formally handled in \cite{Djidjou2018}. The authors showed that a $R_0$ optimization principle holds in the case $\beta_1\equiv\beta_2$, $r_1\not\equiv r_2$ (see \cite[Section 4.2]{Djidjou2018}) and that the pathogen population becomes monomorphic (\textit{i.e.} concentrates around the maximum of the fitness function $\Psi_1+\Psi_2$). In the case $r_1\equiv r_2$, $\beta_1\not\equiv \beta_2$, an invasion fitness function shows which phenotypic value may survive as $\varepsilon \to 0$ (see \cite[Appendix C]{Djidjou2018})). Numerically, a Pairwise Invasibility Plot is used to characterize the evolutionary attractors, leading to either monomorphic or dimorphic (\textit{i.e.} the pathogen concentrates around two phenotypic traits) cases, depending on the fitness functions.} 

\begin{figure}[!h]
\begin{center} 
\includegraphics[width=\linewidth]{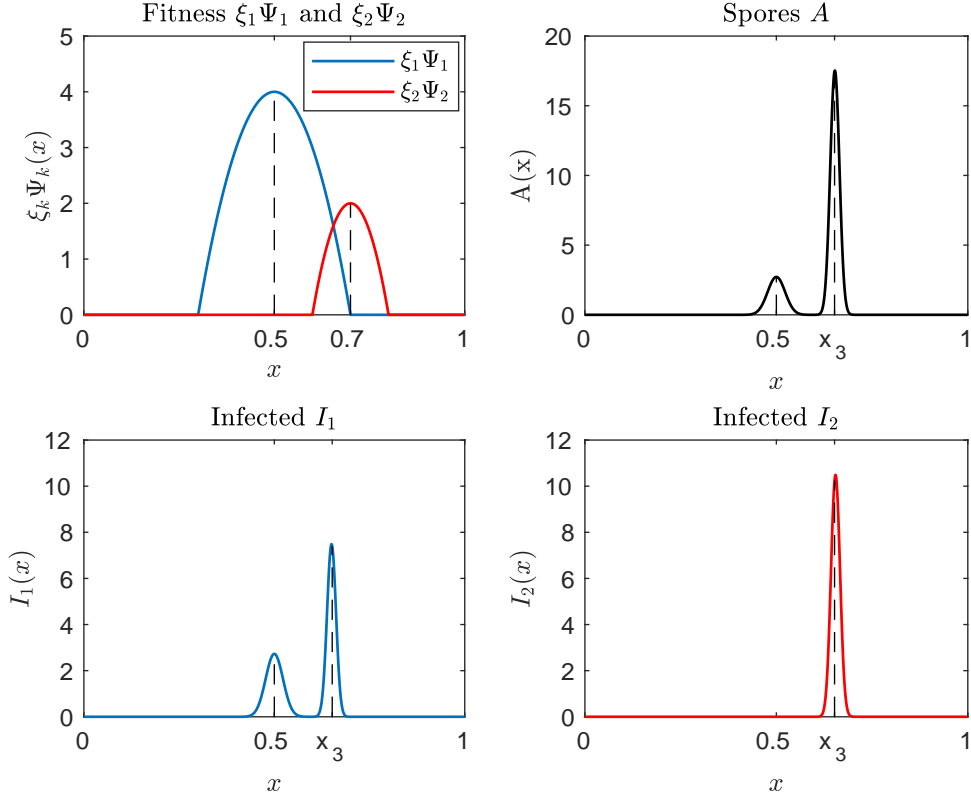}   
\end{center}
\caption{Fitness functions   and endemic equilibrium when Assumption \ref{Assump:supports} does not hold. Though $\Psi_2$ takes its maximum value in $x=0.7$, functions $A$, $I_1$ and $I_2$ concentrate around the trait value $x_3 \simeq 0.652$. {Parameters for this simulation are the same as those of Figure~\ref{Fig1} and ~\ref{Fig2} except for $\beta_1$ and $\beta_2$ which are $\beta_1(x)=200((x-0.3)(0.7-x))^+$  and   $\beta_2(x)=400((x-0.6)(0.8-x))^+$ here.}} \label{Fig3}
\end{figure}

Finally, we are able to prove the uniqueness of the positive equilibrium of \eqref{Eq:Model} given by Theorem \ref{Thm:main}, when $\ep $ is sufficiently small. The case where $\min(R_{0,1}, R_{0,2})=1$ requires an additional assumption on the speed of convergence of the smallest spectral radius as $\varepsilon\to 0$, which is quite natural in our context (it holds for exponentially decaying mutation kernels \cite{Djidjou2017}).
 
\begin{theorem}[Uniqueness of the endemic equilibrium]\label{Thm:uniqueness}
	Let Assumptions \ref{Assump:Psi}, \ref{Assump:supports} and \ref{Assump:gap} be satisfied. Assume moreover that $R_{0,1}>1$ and that  one of the following properties is satisfied: 
	\begin{itemize}
		\item either $R_{0,2}\neq 1$, 
		\item or $R_{0,2}=1 $ and the convergence of $r_\sigma(L_2^\ep)$ towards $R_{0,2}$ is at most polynomial in $\ep$, namely $r_\sigma(L_2^\ep)\leq 1-C\ep ^n$ for some $C>0$, $n>0$.
	\end{itemize}
	Then, for $\ep >0$ sufficiently small, $T^\ep$ has exactly one nonnegative nontrivial fixed point.
\end{theorem}

Our proof is based on a computation of the Leray-Schauder degree in the positive cone of $\Co\left(\Sigma_1\right)\times \Co\left(\Sigma_2\right)$. The use of the Leray-Schauder degree is usually restricted to derive the existence of solutions to nonlinear problems, or to provide lower bounds on the number of solutions; here, we are  able to derive the {\it uniqueness} of solution. Indeed, for $\ep>0$, we show that any equilibrium is {\it stable}, the topological degree provides a way to count the exact number of positive equilibria for the equation, and show uniqueness. Occurrences of such an argument in the literature are scarce but include \cite{Duc-Mad-13} and more recently \cite{LiWuDong2019}.

\section{Proof of Theorem \ref{Thm:main}} \label{Sec:Thm1}

This section is devoted to the proof of Theorem \ref{Thm:main}. To do so, we investigate some dynamical properties of the nonlinear operator $T^\ep$ defined in \eqref{T}. The existence of a nontrivial fixed point follows from the theory of global attractors while the non-existence follows from comparison arguments. Throughout this section we fix $\ep>0$. Set $\Psi=\xi_1\Psi_1+\xi_2\Psi_2$, $\Omega\subset \R^N$ the open set given by
$$
\Omega:=\{x\in\R^N:\;\Psi(x)>0\},
$$
and let us denote by $\chi_A$ the characteristic function of a set $A$.

We split this section into two parts. Section \ref{nonexistence} is devoted to the proof of Theorem \ref{Thm:main} $(i)$,  namely the non-existence of a nontrivial fixed point when $r_\sigma(L^\ep)\leq 1$. In Section \ref{existence} we prove the existence of a nontrivial solution when $r_\sigma(L^\ep)>1$. 

\subsection{Proof of Theorem \ref{Thm:main} $(i)$}\label{nonexistence}

Recall that $\ep>0$ is fixed. To prove the first part of the theorem, we suppose that $r_\sigma(L^\ep)\leq 1$ and denote by ${L^\ep}_{\mid L^1(\Omega)}$ the operator defined for every $\varphi\in L^1(\Omega)$ by:
\begin{equation*}
L^\ep_{\mid L^1(\Omega)}(\varphi)(x)=  
L^\ep \tilde{\varphi}(x), \text{ a.e. } x\in \Omega;  
\end{equation*}
where $\tilde{\varphi}\in L^1(\R^N)$ is defined by
\begin{equation*}
\tilde{\varphi}(x)=\begin{cases}
\varphi(x) &\text{ if } x\in \Omega; \\
0 &\text{ if } x\in \R^N\setminus \Omega.
\end{cases}
\end{equation*}
Lemma \ref{Lemma:Prel_1} then applies and ensures that the operator ${L^\ep}_{\mid L^1(\Omega)}\in \L(L^1(\Omega))$ is positivity improving, compact, has a positive spectral radius and satisfies
$$r_\sigma\left({L^\ep}_{\mid L^1(\Omega)}\right)= r_\sigma(L^\ep).$$
Next using Lemma \ref{Lemma:Prel_2} (1), we have
\begin{equation}\label{Eq:Lim_Leps}
\lim_{n\to \infty}\left\|\dfrac{({L^\ep}_{\mid L^1(\Omega)})^n(\varphi)}{\left(r_\sigma\left({L^\ep}_{\mid L^1(\Omega)}\right)\right)^n}-\Pi(\varphi)\right\|_{L^1(\Omega)}=0,\;\forall \varphi\in L^1(\Omega),
\end{equation}
where $\Pi$ denotes the spectral projection associated to ${L^\ep}_{\mid L^1(\Omega)}$ onto
$$\Ker \left(I-\dfrac{({L^\ep}_{\mid L^1(\Omega)})}{\left(r_\sigma\left({L^\ep}_{\mid L^1(\Omega)}\right)\right)}\right).$$
Let $A\in L^1_+(\R^N)$ be a fixed point of $T^\ep$. To prove Theorem \ref{Thm:main} $(i)$, let us show that $A\equiv 0$. 
To that aim note that we have
\begin{equation}\label{Eq:Fix_Ineq}
	{A}_{\mid \Omega}=\chi _{\Omega}(T^\ep)^n({A})\leq ({L^\ep}_{\mid L^1(\Omega)})^n({{A}_{\mid \Omega}}),\;\forall n\geq 0.
\end{equation}

Now let us observe that, under the stronger assumption that   $r_\sigma(L^\ep)<1$, then Lemma \ref{Lemma:Prel_2} applies and shows 
\begin{equation*}
\lim_{n\to \infty}\left\|\left({L^\ep}_{\mid L^1(\Omega)}\right)^n(A)\right\|_{L^1(\Omega)}=0.
\end{equation*}
Hence $\|{A}\|_{L^1(\Omega)}=0$ and therefore $A=T^\ep(A)=0$ a.e. in $\R^N$.
This completes the proof of the result when $r_\sigma(L^\ep)<1$.\medskip

We now consider the limit case $r_\sigma(L^\ep)=1$. To handle this case let us recall that $\Pi \left(A_{\mid \Omega}\right)\in \Ker \left(I-({L^\ep}_{\mid L^1(\Omega)})\right)$. This allows us to decompose and estimate \eqref{Eq:Fix_Ineq} as follows:
\begin{flalign}\label{Eq:Ineq_P0}
	\Pi({A}_{\mid \Omega})+(I-\Pi)({A}_{\mid \Omega})&=\chi _{\Omega} (T^\ep)^n({A})\nonumber \\
&\leq ({L^\ep}_{\mid L^1(\Omega)})^n(\Pi({A}_{\mid \Omega})+(I-\Pi)({A}_{\mid \Omega})) \\
\nonumber &\leq \Pi({A}_{\mid \Omega})+({L^\ep}_{\mid L^1(\Omega)})^n(I-\Pi)({A}_{\mid \Omega}),
\end{flalign}
for every $n\geq 0$. This leads to 
\begin{equation}\label{contraction}
(I-\Pi)({A}_{\mid \Omega})\leq ({L^\ep}_{\mid L^1(\Omega)})^n(I-\Pi)({A}_{\mid \Omega}),\;\forall n\geq 1.
\end{equation}
{
We now denote $\varphi:=(I-\Pi)(A_{\mid \Omega})\in L^1(\Omega)$. Since $\Pi(\varphi)=0$, it follows from \eqref{contraction} that
$$\|\varphi\|_{L^1(\Omega)}\leq \|(L^\ep_{\mid L^1(\Omega)})^n(\varphi)-\Pi(\varphi)\|_{L^1(\Omega)}, \;\forall n\geq 1.$$
Letting $n\to \infty$ and using \eqref{Eq:Lim_Leps}, we get:
$$
\|(I-\Pi)(A_{\mid \Omega})\|_{L^1(\Omega)}=\|\varphi\|_{L^1(\Omega)}=0.$$
Therefore we have $A_{|\Omega}=\Pi A_{|\Omega}$ and $A_{|\Omega} $ is an eigenvector of $L^\varepsilon_{L^1(\Omega)}$ associated with the eigenvalue $r_\sigma(L^\varepsilon_{|L^1(\Omega)})=1 $, that is $L^\varepsilon_{|L^1(\Omega)} A_{|\Omega} = A_{|\Omega}$.} Recalling \eqref{Eq:Ineq_P0} this yields
$${A}_{\mid \Omega}=\chi_{\Omega}T^\ep({A})=({L^\ep}_{\mid L^1(\Omega)})({A}_{\mid \Omega}),$$
and this ensures that
$$\int_{\R^N}\beta_k(z){A}_{\mid \Omega}(z)\dd z=0 \quad \forall k\in\{1,2\},$$
and therefore ${A}_{\mid \Omega}=0$ a.e. in $\Omega$ (recall $\beta_1+\beta_2>0$ on $\Omega$ by definition \eqref{Eq:Psi_k}).
The equation $A=T^\ep(A)$ ensures that $A=0$ a.e. in $\R^N$, that completes the proof of Theorem \ref{Thm:main} $(i)$.\qed

\subsection{Proof of Theorem \ref{Thm:main} $(ii)$}\label{existence}

We now turn to the proof of the existence of a nontrivial fixed point for the nonlinear operator $T^\ep$. To that aim we shall make use of the theory of global attractors and uniform persistence theory for which we refer to \cite{MagalZhao2005}. To perform our analysis and prove the theorem we define the sets
\begin{equation}\label{Eq:M}
\M_0:=\left\{\varphi\in L^1_+(\R^N):\; \int_{\Omega}\varphi(y)\dd y>0\right\}\text{ and } \partial \M_0=\{\varphi\in L^1_+(\R^N):\;\chi_\Omega \varphi=0\; a.e.\},
\end{equation}
so that 
$$
L_+^1(\R^N)=\M_0\cup \partial\M_0.
$$
Note also that we have the following invariant properties
\begin{equation*}
T^\ep(\M_0)\subset \M_0\text{ and }T^\ep(\partial \M_0)=\{0_{L^1}\}\subset \partial\M_0.
\end{equation*}
Next let us observe that $T^\ep$ is bounded on $L^1_+(\R^N)$. Indeed, recalling the definition of $\Psi_k$ in \eqref{Eq:Psi_k} it is readily checked that
\begin{equation}\label{item:pointdiss} 
\|T^\ep(\varphi)\|_{L^1(\R^N)}\leq \dfrac{\Lambda }{\delta \theta}\left[\xi_1 \|r_1\|_{L^\infty}+\xi_2 \|r_2\|_{L^\infty}\right],\;\forall \varphi\in L^1_+(\R^N).
\end{equation}
Our first lemma deals with the weak persistence of $T^\ep$ and  $T^\ep_k$ as defined in \eqref{eq:Tk}. Our result reads as follows.

\begin{lemma}\label{Lemma:Prel_3}
	Let Assumption \ref{Assump:Psi} be satisfied. 
	\begin{enumerate}
		\item If $r_\sigma(L_k^\ep)>1$ for some $k\in\{1,2\}$, then we have 
			\begin{equation}\label{unifpers-sep}
				\limsup_{n\to \infty}\int_{\R^N} \beta_k(y) (T^\ep_k)^n(\varphi)(y)\dd y\geq \dfrac{\theta}{2}\left(r_\sigma(L_k^\ep)-1\right),\;\forall \varphi\in \M_0. 
			\end{equation}
		\item If $r_\sigma(L^\ep)>1$, then there exists $k\in\{1,2\}$ such that
			\begin{equation}\label{unifpers}
				\limsup_{n\to \infty}\int_{\R^N} \beta_k(y) (T^\ep)^n(\varphi)(y)\dd y\geq \dfrac{\theta}{2}\left(r_\sigma(L^\ep)-1\right),\;\forall \varphi\in \M_0.
			\end{equation}
		\item {Let $k\in \{1,2\}$ and assume that $r_{\sigma}(L_k^\varepsilon)>1$. If $A\in \mathcal M_0$  is a fixed point of $T^\varepsilon$, i.e. $T^\varepsilon(A)=A$, then we have
			\begin{equation}\label{eq:lowerbound}
				\int_{\mathbb R^N}\beta_k(y)A(y)\dd y \geq \frac{\theta}{2}(r_\sigma(L^\varepsilon_k)-1). 
			\end{equation}}
	\end{enumerate}
\end{lemma}

\begin{proof}
Let us first show \eqref{unifpers}.
We argue by contradiction by assuming that there exists $\varphi\in \M_0$ such that
	\begin{equation*}
		\limsup_{n\to \infty} \int_{\R^N}\beta_k(y)(T^\ep)^n(\varphi)(y)\dd y<\dfrac{\theta}{2}(r_\sigma(L^\ep)-1)=:\eta,\qquad k=1,2.
	\end{equation*}
Then, there exists an integer $n_0\geq 1$ such that
\begin{equation*}
	\int_{\R^N}\beta_k(y)(T^\ep)^{n}(\varphi)(y)\dd y\leq\eta,\qquad  \text{ for }k=1,2\text{ and }n\geq n_0
\end{equation*}
	therefore
	\begin{equation*}
		(T^\ep)^{n_0+1} (\varphi)(x)\geq\left(\dfrac{\theta }{\theta+\eta}L^\ep\right) ((T^\ep)^{n_0}(\varphi))(x),\qquad\text{for a.e. } x\in \R^N,
	\end{equation*}
	and by induction
	\begin{equation}\label{Eq:Weak_Persist}
		(T^\ep)^{n_0+n} (\varphi)\geq \left(\dfrac{\theta}{\theta+\eta}L^\ep\right)^n ((T^\ep)^{n_0}(\varphi))(x)
	\end{equation}
	for a.e. $x\in \Omega $ and for every $n\geq 1$.
Next set 
	$$\tilde{\varphi}=\left((T^\ep)^{n_0}(\varphi)\right)_{\mid \Omega}\in L^1_+(\Omega)\backslash\left\{0\right\}.$$
	By Lemma \ref{Lemma:Prel_1}, the operator
	${\left(\dfrac{\theta}{\theta+\eta} L^\ep\right)}_{\mid L^1(\Omega)}\in \L(L^1(\Omega))$
	is positivity improving, compact and satisfies
	\begin{equation*}
		r_\sigma\left({\left(\dfrac{\theta}{\theta+\eta} L^\ep\right)}_{\mid L^1(\Omega)}\right)=\dfrac{\theta}{\theta+\eta} r_\sigma(L^\ep)=\dfrac{2r_\sigma(L^\ep)}{1+r_\sigma(L^\ep)}>1
	\end{equation*}
	since $r_\sigma(L^\ep)>1$. Applying Lemma \ref{Lemma:Prel_2} yields 
	$$\lim_{n\to \infty}\left\|\left({\left(\dfrac{\theta}{\theta+\eta}L^\ep\right)}_{\mid L^1(\Omega)}\right)^n (\tilde{\varphi})\right\|_{L^1(\Omega)}=\infty,$$
	so that \eqref{Eq:Weak_Persist} ensures that the sequence $\|(T^\ep)^n(\varphi)\|_{L^1(\Omega)}$ is unbounded. This contradicts the point dissipativity of $T^\ep$ as stated in \eqref{item:pointdiss}. The proofs of \eqref{unifpers-sep} for $T_1^\ep$ and $T_2^\ep$ are similar.\medskip

	{Next we show \eqref{eq:lowerbound}. Let $A\in\mathcal M_0$ be a fixed point of $T^\varepsilon$,
 	we assume by contradiction that 
	$$\int_{\mathbb R^N}\beta_k(y)A(y)\dd y\leq \frac{\theta}{2}(r_\sigma(L_k^\varepsilon)-1)=: \eta.$$ 
	Then, since  $(T^\varepsilon)^n A=A$ for all $n\geq 0$ we have
	\begin{equation*}
		\int_{\mathbb R^N}\beta_k(y)\big[(T^\varepsilon)^n A\big](y)\dd y \leq \eta.
	\end{equation*}
	In particular the following inequalities hold: 
	\begin{align*}
		T^\varepsilon(A) & = T_1^\varepsilon(A)+T_2^\varepsilon(A)\geq T_k^\varepsilon(A)\geq \frac{\theta}{\theta+\eta}L_k^\varepsilon A \\ 
		(T^\varepsilon)^2(A) & = T_1^\varepsilon(T^\varepsilon A) + T_2^\varepsilon(T^\varepsilon A) \geq T_k^\varepsilon (T^\varepsilon A)\geq \frac{\theta}{\theta+\eta} L_k^\varepsilon \left(T^\varepsilon A\right) \geq \left(\frac{\theta}{\theta+\eta} L_k^\varepsilon\right)^2 A \\
		& \quad \vdots\\
		(T^\varepsilon)^n A& = T_1^\varepsilon (T^\varepsilon)^{n-1}A + T_2^\varepsilon(T^\varepsilon)^{n-1}A\geq T_k^\varepsilon(T^\varepsilon)^{n-1}A \geq \frac{\theta}{\theta+\eta} L_k^\varepsilon (T^\varepsilon)^{n-1} A \\
		&\geq \left(\frac{\theta}{\theta+\eta} L_k^\varepsilon\right)^{n} A
	\end{align*}
	By Lemma \ref{Lemma:Prel_1}, the operator
	${\left(\dfrac{\theta}{\theta+\eta} L_k^\ep\right)}_{\mid L^1(\Omega)}\in \L(L^1(\Omega))$
	is positivity improving, compact and satisfies
	\begin{equation*}
		r_\sigma\left({\left(\dfrac{\theta}{\theta+\eta} L_k^\ep\right)}_{\mid L^1(\Omega)}\right)=\dfrac{\theta}{\theta+\eta} r_\sigma(L_k^\ep)=\dfrac{2r_\sigma(L_k^\ep)}{1+r_\sigma(L_k^\ep)}>1
	\end{equation*}
	since $r_\sigma(L^\ep)>1$. Applying Lemma \ref{Lemma:Prel_2} leads to a contradiction. 

	Lemma \ref{Lemma:Prel_3} is proved.\qedhere
	}
\end{proof}

We are now able to complete the proof of Theorem \ref{Thm:main} $(ii)$.\\
\begin{proof}[Proof of Theorem \ref{Thm:main} $(ii)$]
Recall that throughout this section, $\ep>0$ is fixed. Assume that $r_\sigma(L^\ep)>1$. As $0 \le T^\ep \le L^\ep$ and as $L^\ep$ is compact (see Lemma \ref{Lemma:Prel_1}), then  $T^\ep$ is bounded and compact.  Now Theorem 2.9 in \cite{MagalZhao2005} applies and ensures that there is a compact global attractor $\mathcal A\subset L^1_+(\R^N)$ for $T^\ep$, {\it i.e.} $\mathcal A$ attracts every bounded subset of $L^1_+(\R^N)$ under the iteration of $T^\ep$. 
	Next by Lemma \ref{Lemma:Prel_3}, $T^\ep$ is weakly uniformly persistent with respect to the decomposition pair $\left(\M_0,\partial\M_0\right)$ of the state space $L_+^1(\R^N)$. Therefore {\cite[Proposition 3.2]{MagalZhao2005}} applies and ensures that $T^\ep$ is also strongly uniformly persistent with respect to this decomposition, {\it i.e.}  there exists $\kappa>0$ such that
 	\begin{equation*}
		\liminf_{n\to+\infty}\Vert (T^\ep)^n(\varphi)\Vert_{L^1(\Omega)}\geq \kappa,\;\forall \varphi\in \M_0.
	\end{equation*}
	As a consequence, according to \cite{MagalZhao2005}, $T^\ep_{\mid \M_0}$ admits a compact global  attractor $\mathcal A_0\subset \M_0$ and $T^\ep$ has at least one fixed point $A\in \mathcal A_0$. From the equation $A=T^\ep(A)$, it is readily checked that $A>0$ a.e. and belongs to $L^\infty(\R^N)$, while the uniform boundedness (with respect to $\ep$) of such a fixed point follows from \eqref{item:pointdiss}.
	
Finally, it remains to prove the continuity of the fixed point $A$. The facts that $\Psi_k A\in L^1(\R^N)$ for each $k=1,2$ and $m_\ep\in L^\infty(\R^N)$, imply (see \textit{e.g.} \cite[Corollary 3.9.6, p. 207]{Bogachev2007}) that $m_\ep \star (\Psi_k A) \in \Co(\R^N)$. From the expression \eqref{T} of $T^\ep$, it follows that $A\in \Co(\R^N)$.	This completes the proof of Theorem \ref{Thm:main} $(ii)$.
\end{proof}
%
%
%
%
%
%
%

\section{Proof of Theorem \ref{Thm:shape}}\label{Sec:Thm2}

In this section, we investigate the shape of the endemic equilibria and we prove Theorem \ref{Thm:shape}. Hence we assume throughout this section that Assumptions \ref{Assump:Psi},  \ref{Assump:supports} and \ref{Assump:gap} hold. We furthermore assume that
$$
R_0=\max\{R_{0,1},R_{0,2}\}>1.
$$
Next recall that since $r_\sigma(L^\ep)\to R_0$ as $\ep\to 0$, Theorem \ref{Thm:main} implies that Problem \eqref{Eq:AepsIntro} has at least a nontrivial fixed point for all $\ep$ sufficiently small. We denote by $A^\ep\in L_+^1(\R^N)$ such a nontrivial fixed point of $T^\ep$, for all $\ep$ small enough. It is not difficult to check that $A^\ep>0$ a.e. 

Recalling the definition of the open sets
$$
\Omega_k=\{x\in\R^N:\;\Psi_k(x)>0\},\;k=1,2, \qquad \Omega=\Omega_1\sqcup \Omega_2=\{x\in\mathbb R^N:\;\Psi(x)>0\},
$$
note that Assumption \ref{Assump:supports} ensures that there exists $\eta>0$ such that $\|x-y\|\geq \eta$ for all $(x,y)\in \Omega_1\times\Omega_2$. In what follows the functions $\chi_{\Omega_k}$ denotes the characteristic functions for $\Omega_k$.

Throughout this section, for all $\ep>0$ small enough, $A^\ep\in L^1_+(\mathbb R^N)\setminus \{0\}$ denotes a positive solution to the equation: 
\begin{equation}\label{Eq:A}
	A^\ep=\frac{\Lambda\xi_1 (m_\ep\star\Psi_1A^\ep)}{\theta+\int_{\mathbb R^N}\beta_1(z)A^\ep(z)\dd z}+\frac{\Lambda\xi_2 (m_\ep\star\Psi_2A^\ep)}{\theta+\int_{\mathbb R^N}\beta_2(z)A^\ep(z)\dd z}.
\end{equation}

\subsection{Preliminary estimates}

Recall the definition of $L_k^\ep$ in \eqref{Eq:Lk_ep}. Let  $\phi^{\ep,1}_k\in L^1_+(\R^N)$ with $\phi^{\ep,1}_k>0$ and $\|\phi_k^{\ep,1}\|_{L^1(\R^N)}=1$ be the principal eigenvector of $L^\ep_k$ associated to its principal eigenvalue, which is equal to the spectral radius $r_\sigma(L_k^\ep)$.
We now recall some results related to the one host model. We refer to \cite{Djidjou2017} for more details (see also Lemma \ref{Lemma:Prel_1}).
\begin{lemma}\label{LE-one-host}
Let Assumption \ref{Assump:Psi} be satisfied. Let $k \in \{1,2\}$ and $\ep>0$ be given and assume that $r_\sigma(L_k^\ep)>1$.
Then the equation
\begin{equation*}
A_k=\dfrac{\Lambda\xi_k \chi_{\Omega_k} (m_{\ep }\ast (\Psi_k A_k))}{\theta+\int_{\R^N}\beta_k(z) A_k(z)\dd z},\;\qquad A_k\in L^1_+(\mathbb R^N)\setminus\{0\},
\end{equation*}
has a unique solution, given by
\begin{equation}\label{nuk}
	A^{\ep, *}_k=\nu_k^\ep \phi_k^{\ep,1}\text{ with }\nu^\ep_k=\dfrac{\theta(r_\sigma(L_k^\ep)-1)}{\int_{\R^N}\beta_k(z)\phi^{\ep,1}_k(z)\dd z}.
\end{equation}
\end{lemma}

Now, using the separation assumption on the sets $\Omega_k, k\in\{1,2\}$ and the decay at infinity of $m$, we derive the following preliminary lemma that will be used to prove Theorem \ref{Thm:shape} in the next subsection.

\begin{lemma}\label{Lemma:Decay}
Suppose that Assumptions \ref{Assump:Psi} and \ref{Assump:supports} are satisfied. Then, for each $(k,l)\in\{1,2\}^2$ with $k\neq l$, the following properties hold:
\begin{enumerate}
\item[(a)]  we have 
	\begin{equation*}
		\int_{\mathbb R^N}\chi_{\Sigma_l}\phi^{\ep,1}_k(z)\dd z=o(\ep^\infty), \qquad \nu^\ep_k\int_{\mathbb R^N}\chi_{\Sigma_l} \phi^{\ep,1}_k(z)\dd z=o(\ep^\infty),
	\end{equation*}
		for all $\ep\ll 1$, where $\Sigma_l$ and $\nu_k^\ep$ are respectively defined in  \eqref{support} and \eqref{nuk}.

\item[(b)] Let $p\in [1,\infty)$ be given. Then, for any $A\in L^1_+(\mathbb R^N)$, the following estimate holds
\begin{equation}\label{Eq:Decay_G}
	\Vert \chi_{\Sigma_k} m_\ep*(\Psi_l A)\Vert_{L^p(\mathbb R^N)}=\|A\|_{L^1(\Omega_l)}\times o(\ep^\infty),
\end{equation}
		where the term $o(\ep^\infty)$ is independent of $A\in L^1_+(\mathbb R^N)$.
\end{enumerate}
\end{lemma}
\begin{proof}
We first prove $(a)$. To that aim let us first notice that, due to Assumption \ref{Assump:supports}, there exists $\eta>0$ such that $\|x-y\|\geq \eta$ for all $(x,y)\in \Sigma_1\times\Sigma_2$. Thus, due to the decay assumption for $m$ at infinity, one obtains
\begin{equation}\label{small}
m_\ep(x-y)=o(\ep^\infty),
\end{equation}
	uniformly for $(x,y)$ in the compact set $\Sigma_1\times\Sigma_2$. Now let $(k,l)\in\{1,2\}^2, k\neq l$ be given. By the definition of $\phi^{\ep,1}_k$ we have
\begin{equation}\label{ei}
\phi^{\ep,1}_k=\dfrac{\Lambda \xi_k}{\theta r_\sigma\left(L^{\ep}_k\right)}\left(m_\ep \star (\Psi_k \phi^{\ep,1}_k)\right).
\end{equation}
	Integrating \eqref{ei} over $\Sigma_l$ and recalling that $\|\phi_k^{\ep,1}\|_{L^1(\R^N)}=1$ we get
$$
	\int_{\mathbb R^N}\chi_{\Sigma_l}\phi^{\ep,1}_k(z)\dd z \leq \dfrac{\Lambda \xi_k |\Sigma_l| }{\theta r_\sigma\left(L^{\ep}_k\right)}\|\Psi_k\|_{L^\infty(\R^N)} \sup_{(x,y)\in \Sigma_k\times \Sigma_l} m_\ep(x-y).
$$
Since $r_\sigma(L_k^\ep)\to R_{0,k}>0$ as $\ep\to 0$ (recalling \eqref{Eq:R0,k}), this yields
$$
	\int_{\mathbb R^N}\chi_{\Sigma_l}\phi^{\ep,1}_k(z)\dd z=o(\ep^\infty)\text{ as }\ep\to 0,
$$
and completes the proof of the first estimate in $(a)$. Next coming back to \eqref{ei} and recalling the definition of $\Psi_k$ in \eqref{Eq:Psi_k} we get for all $x\in\R^N$
\begin{equation*}
	\phi^{\ep,1}_k(x)\leq \dfrac{\Lambda \xi_k}{\theta r_\sigma\left(L^{\ep}_k\right)}\frac{\|m_\ep\|_{L^\infty}\|r_k\|_{L^\infty}}{\delta\theta}\int_{\mathbb R^N} \beta_k(y) \phi^{\ep,1}_k(y)\dd y.
\end{equation*}
Hence since $\|\phi_k^{\ep,1}\|_{L^1(\R^N)}=1$ and $\|m_\ep\|_{L^\infty}=\mathcal O(\ep^{-N})$, integrating the above inequality over the bounded set $\Sigma_k$,  there exists a constant $C>0$ such that 
\begin{equation*}
\int_{\Sigma_k} \beta_k(y) \phi^{\ep,1}_k(y)\dd y\geq C\ep^N,\;\forall \ep\ll 1.
\end{equation*}
Hence we get
$$
	\nu_k^\ep=\mathcal{O}(\ep^{-N})\text{ as }\ep\to 0,
$$
and the second estimate in $(a)$ follows. We now turn to the proof of $(b)$. 
Let $A\in L^1_+(\mathbb R^N)$ be given. Then we have, for all $\ep>0$,
\begin{equation*}
\left\vert m_\ep \star (\Psi_2 A)(x)\right\vert \leq \sup_{(y,z)\in \Sigma_1\times \Sigma_2} m_\ep(y-z)\;\;\|\Psi_2\|_{L^\infty} \|A\|_{L^1(\Omega_2)},\;\forall x\in \Sigma_1.
\end{equation*}
Hence, integrating the above inequality on the compact set $\Sigma_1$, we obtain that, for all $p\in [1,\infty)$:
\begin{equation*}
	\|\chi_{\Sigma_1}m_\ep \star (\Psi_2 A)(x)\|_{L^p(\mathbb R^N)}\leq |\Sigma_1|^{1/p} \sup_{(y,z)\in \Sigma_1\times \Sigma_2} m_\ep(y-z)\;\;\|\Psi_2\|_{L^\infty} \|A\|_{L^1(\Omega_2)},
\end{equation*}
and the estimate with $k=1$ and $l=2$ follows recalling \eqref{small}.
The other estimate interchanging the index $1$ and $2$ is similar.
This completes the proof of $(b)$.
\end{proof}

\subsection{Proof of Theorem \ref{Thm:shape}}

This section is devoted to the proof of Theorem \ref{Thm:shape}. 
Throughout this section we assume that $R_0>1$ so that, since $r_\sigma(L^\ep)\to R_0$ as $\ep\to 0$, there exists $\ep_0>0$ such that Problem \eqref{Eq:AepsIntro} has a nontrivial fixed point $A^\ep$ for each $\ep\in (0,\ep_0]$ (see Theorem \ref{Thm:main}).
Recall that since $T^\ep$ is bounded with respect to $\ep$, there exists $M>0$ such that
$$
\|A^\ep\|_{L^1(\R^N)}\leq M,\;\forall \ep\in (0,\ep_0].
$$
As before, set 
\begin{equation*}
A_k^\ep=\chi_{\Omega_k} A^\ep,\;k=1,2,\;\ep\in (0,\ep_0],
\end{equation*}
and observe that $\|A_1^\ep\|_{L^1(\Omega_1)}+\|A_2^\ep\|_{L^1(\Omega_2)}\leq M$ for all $\ep\in (0,\ep_0]$. Now let us define
\begin{equation}\label{Eq:Def_Mu}
\mu^\ep_k=\theta+\displaystyle \int_{\R^N}\beta_k(z)A^\ep(z)\dd z, \quad \forall k\in\{1,2\}, \quad \forall \ep\in(0,{\ep_0}]
\end{equation}
as well as
$$\Psi^\ep:=\dfrac{\Lambda \xi_1 \Psi_1}{\mu^\ep_1}+\dfrac{\Lambda \xi_2 \Psi_2}{\mu^\ep_2}.$$
With these notations, note that $A^\ep$ becomes a positive fixed point for the linear operator $K^\ep\in \L(L^1(\R^N))$ defined by
\begin{equation*}
	K^\ep\varphi :=m_\ep\star \left(\Psi^\ep\varphi\right), \qquad \varphi\in L^1(\R^N).
\end{equation*}
Our first step consists in proving the next lemma.
\begin{lemma}\label{LE1}
The following estimate holds
\begin{equation}\label{Eq:Key}
r_\sigma(L_k^\ep)\leq \dfrac{\mu^\ep_k}{\theta},\;\forall \ep\in (0,\ep_0].
\end{equation}
\end{lemma}
\begin{proof}
Let us first note that
\begin{equation}\label{Eq:Maj_Mu_k}
\theta \leq \mu^\ep_k\leq \theta+M \|\beta_k\|_{L^\infty}, \quad \forall k\in \{1,2\}, \quad \forall \ep\in(0,\ep_0]
\end{equation}
	for some constant $M>0$. Note that since $A^\ep>0$ and $K^\ep A^\ep=A^\ep$, we obtain by a version of the Krein-Rutman theorem (see {\it e.g.} \cite[Corollary 4.2.15 p.273]{MeyerNieberg91}) that
$$r_\sigma(K^\ep)=1,\forall \ep\in(0,\ep_0].$$
	On the other hand, we have $A^\ep=K^\ep A^\ep = \frac{\theta}{\mu_1^\ep}L_1^\ep A^\ep+\frac{\theta}{\mu_2^\ep}L_2^\ep A$, thus for $n\geq 1$ we have 
	\begin{equation*}
		0\leq \left(\frac{\theta}{\mu_1^\ep}L_1^\ep\right)^n A^\ep\leq \left(\frac{\theta}{\mu_1^\ep}L_1^\ep\right)^{n-1} A^\ep\leq\ldots\leq  \frac{\theta}{\mu_1^\ep}L_1^\ep A^\ep\leq A^\ep,
	\end{equation*}
	therefore the contrapositive of  Lemma \ref{Lemma:Prel_2} item \ref{item:opnorm} shows that $\frac{\theta}{\mu_1^\ep}r_\sigma(L_1^\ep)\leq 1$. Similarly, we have $\frac{\theta}{\mu_2^\ep}r_\sigma(L_2^\ep)\leq 1$. 
\end{proof}

We recall that throughout this section the condition $R_0=\max\{R_{0,1},R_{0,2}\}>1$ holds.
We now set $\Theta_k=\sqrt{\Psi_k}$ and we define the self-adjoint operators $S_k^\ep\in \L(L^2(\Omega_k))$ (recall $\Omega_k=\{\Psi_k>0\}$), for $k=1,2$, by
\begin{equation*}
S_k^\ep =\dfrac{\Lambda \xi_k}{\theta} \Theta_k(m_\ep \star (\Theta_k \cdot)).
\end{equation*}
Here recall that $\sigma(S_k^\ep)=\sigma(L_k^\ep)$ since $\Omega_k$ is bounded (see Lemma \ref{Lemma:Prel_1}). Our next lemma reads as follows.
\begin{lemma}\label{LE2}
Let $k\in\{1,2\}$ be such that $R_{0,k}>1$. Then we have
\begin{equation}\label{Eq:Key_dist}
{\rm dist}\left(\dfrac{\mu^\ep_k}{\theta}, \sigma(S^\ep_k)\right)={\rm dist}\left(\dfrac{\mu^\ep_k}{\theta}, \sigma(L^\ep_k)\right)=o(\ep^\infty), \;\forall \ep\ll 1.
\end{equation}
\end{lemma}

\begin{proof}
	Let us assume that $R_{0,1}>1$ (the case $R_{0,2}>1$ is obtained by the symmetry of the problem with respect to the indices).
We recall that
$$
\Omega_k=\{x\in\R^N:\;\Psi_k>0\},\;k=1,2,
$$
	and we denote by $\{\lambda^{\ep,n}_k\}_{n\geq 1}$ the eigenvalues of $S_k^\ep$ (and of $L_k^\ep$) ordered by decreasing modulus, so that $\lambda^{\ep,1}_k=r_\sigma(S_k^\ep)=r_\sigma(L_k^\ep)$. Next multiplying \eqref{Eq:A} by $\frac{\mu_1^\varepsilon \Theta_1}{\theta}$ and using Lemma \ref{Lemma:Decay} $(b)$ yields
$$\dfrac{\mu^\ep_1 \Theta_1 A^\ep}{\theta}-\dfrac{\Lambda \xi_1 \Theta_1 (m_\ep \star (\Psi_1 A^\ep))}{\theta}=\dfrac{\Lambda \xi_2 \mu^\ep_1 \Theta_1 (m_\ep \star (\Psi_2 A^\ep))}{\theta \mu^\ep_2}=o(\ep^\infty),$$
in $L^2(\Omega_1)$. Hence the following estimate holds
\begin{equation}\label{Eq:Norm_petito}
\left\|\left(\dfrac{\mu^\ep_1}{\theta} I-S^\ep_1\right) \Theta_1 A^\ep \right\|_{L^2(\Omega_1)}=o(\ep^\infty).
\end{equation}
On the other hand, since $S^\ep_1$ is self-adjoint, then the following estimate holds (see \textit{e.g.} \cite{ReedSimon72})
$$\left\|\left(\dfrac{\mu^\ep_1}{\theta} I-S^\ep_1\right)^{-1}\right\|_{\L(L^2(\Omega_1))}=\dfrac{1}{\text{dist}\left(\dfrac{\mu^\ep_1}{\theta}, \sigma(S^\ep_1)\right)}.$$
By setting
$$y^\ep_1:=\left(\dfrac{\mu^\ep_1}{\theta} I-S^\ep_1\right)\Theta_1 A^\ep,$$
we get
$$\Theta_1 A^\ep=\left(\dfrac{\mu^\ep_1}{\theta} I-S^\ep_1\right)^{-1}y^\ep_1, \qquad \|y^\ep_1\|_{L^2(\Omega_1)}=o(\ep^\infty),$$
so that
$$\|\Theta_1 A^\ep\|_{L^2(\Omega_1)}\leq \dfrac{\|y^\ep_1\|_{L^2(\Omega_1)}}{\text{dist}\left(\dfrac{\mu^\ep_1}{\theta}, \sigma(S^\ep_1)\right)}$$
and
\begin{equation}
\label{Eq:Dist}
\text{dist}\left(\dfrac{\mu^\ep_1}{\theta}, \sigma(S^\ep_1)\right)\leq \dfrac{\|y^\ep_1\|_{L^2(\Omega_1)}}{\|\Theta_1 A^\ep\|_{L^2(\Omega_1)}}\leq \dfrac{\sqrt{|\Omega_1|}\|y^\ep_1\|_{L^2(\Omega_1)}}{\|\Theta_1 A^\ep\|_{L^1(\Omega_1)}},
\end{equation}
	where we have used the Cauchy-Schwarz inequality in $L^2(\Omega_1)$.

To complete the proof of the lemma, we  show that the quantity $\|\Theta_1 A^\ep\|_{L^1(\Omega_1)}$ does not become too small when $\ep\to 0$.
	{Since $A^\varepsilon$ is a fixed point of $T^\varepsilon$, it follows from \eqref{eq:lowerbound} in Lemma \ref{Lemma:Prel_3} that 
	\begin{equation*}
	\int_{\mathbb R^N}\beta_1(y)A^\varepsilon(y)\dd y \geq \frac{\theta}{2}(r_\sigma(L_1^\varepsilon)-1),
	\end{equation*}
therefore we get for $\ep$ sufficiently small:
\begin{align*}
	\|\beta_1\|_{L^\infty}\|\chi_{\Sigma_1} A^\ep\|_{L^1(\R^N)}&\geq \int_{\R^N}\beta_1(y)A^\ep(y)\dd y \geq \dfrac{\theta}{2}\left(r_\sigma(L^\ep_1)-1\right).
\end{align*}}
	Next multiplying  \eqref{Eq:A} by $\frac{\mu_1^\varepsilon\chi_{\Sigma_1}}{\theta}$ and integrating leads us to 
$$\dfrac{\Lambda \xi_1}{\theta}\|\Theta_1\|_{L^\infty}\|\Theta_1 A^\ep\|_{L^1(\Omega_1)}\geq \dfrac{\mu^\ep_1 \|\chi_{\Sigma_1} A^\ep\|_{L^1(\mathbb R^N)}}{\theta}-\dfrac{\Lambda \xi_2 \mu^\ep_1}{\theta \mu^\ep_2}\iint_{\Omega_1\times\Omega_2} m_\ep(x-y)\Psi_2(y)A^\ep(y)\dd y \dd x,$$
while Lemma \ref{Lemma:Decay} ensures that
$$\iint_{\Omega_1\times \Omega_2}m_\ep(x-y)\Psi_2(y)A^\ep(y)\dd y \dd x=o(\ep^\infty).$$
As a consequence  there exist $\overline{\ep}>0$ and $\eta>0$ such that
\begin{equation*}
\|\Theta_1 A^\ep\|_{L^1(\Omega_1)}\geq \eta,\;\forall \ep\in (0,\overline\ep].
\end{equation*}
The latter estimate combined with \eqref{Eq:Dist} completes the proof of the lemma. 
\end{proof}

As a corollary of the above lemma, we also have the following result.
\begin{corollary}\label{CO1}
Let $k\in\{1,2\}$ be such that $R_{0,k}>1$. 
Then the following holds true for $\ep>0$ sufficiently small
\begin{equation}\label{Eq:Mu}
\dfrac{\mu^\ep_k}{\theta}=\lambda^{\ep,1}_k+o(\ep^\infty).
\end{equation}
\end{corollary}
\begin{proof}
	Here we consider the case where $R_{0,1}>1$. The case where $R_{0,2}>1$ is obtained similarly. 

	In view of Lemma \ref{LE2}, {the distance between $\frac{\mu^\varepsilon_1}{\theta} $ and the spectrum of $L_1^\varepsilon$ (which consists in the ordered sequence of eigenvalues $(\lambda_1^{\varepsilon, n})_{n\geq 1}$) is controlled by $o(\varepsilon^\infty)$. To show that $\frac{\mu_1^\varepsilon}{\theta} $ is actually \mbox{$\varepsilon^\infty$-close} to $\lambda_1^{\varepsilon, 1}$ and not to another location of the spectrum,} we argue by contradiction and assume that there exist a sequence $\{\ep_i\}\subset (0,\infty)$ going to $0$ as $i\to\infty$ and a sequence {$n_i\in\mathbb N,$ $n_i>1$} such that for all $i$ one has
\begin{equation*}
\dfrac{\mu^{\ep_i}_1}{\theta}=\lambda^{\ep_i,n_i}_1+o(\ep_i^\infty).
\end{equation*}
Firstly we have
\begin{equation*}
\dfrac{\mu^{\ep_i}_1}{\theta}=\lambda^{\ep_i,n_i}_1+o(\ep_i^\infty)\leq \lambda^{\ep_i,2}_1+o(\ep_i^\infty), \qquad \forall i\geq 0. 
\end{equation*}
Next using Assumption \ref{Assump:gap} one has $\lambda^{\ep_i,1}_1-\lambda^{\ep_i,2}_1\geq c\ep_i^{n_1}$ for all $k$ large enough,  where $c>0$ and $n_1\in\mathbb N$ are given constants independent of $i$.
This yields
\begin{equation*}
	\dfrac{\mu^{\ep_i}_1}{\theta}-{r_\sigma(L_1^{\ep_i})}= \dfrac{\mu^{\ep_i}_1}{\theta}-\lambda^{\ep_i,1}_1\leq -c\ep_i^{n_1}+o(\ep_i^\infty),\;\forall i\gg 1.
\end{equation*}
	This contradicts the estimate provided by Lemma \ref{LE1} and Corollary \ref{CO1} is proved.
\end{proof}

Our next lemma describes the asymptotic shape as $\ep\to 0$ of the fixed points in the domain $\Omega_k$, when $R_{0, k}>1$.
\begin{lemma}\label{LE3}
Let $k\in\{1,2\}$ such that $R_{0,k}>1$ and $A^\varepsilon$ be a positive solution to $T^\varepsilon A^\varepsilon=A^\varepsilon$. Then, the following estimate holds for $\ep>0$ sufficiently small:
\begin{equation}\label{Eq:Cv-Ak}
\left\|{A^\ep}-\nu^\ep_k {\phi^{\ep,1}_k}\right\|_{L^2(\Omega_k)}=o(\ep^\infty),
\end{equation}
	where  $\nu_k^\ep$  is defined in \eqref{nuk}.
\end{lemma}

\begin{proof}
	Here we only deal with the case $R_{0,1}>1$, the case  $R_{0,2}>1$ being similar.

	We first remark that, by definition of $\Psi_1$ (see \eqref{Eq:Psi_k}),  $\Omega_1\subset \Sigma_1$ and  $\Psi_1=\Theta_1=0$ on $\Sigma_1\setminus\Omega_1$. Observe that Corollary \ref{CO1} together with \eqref{Eq:Norm_petito} yields
\begin{equation}\label{Eq:S-Lambda}
\|(\lambda^{\ep,1}_1I -S^\ep_1 )\Theta_1 A^\ep\|_{L^2(\Omega_1)}=o(\ep^\infty),\; \ep\ll 1.
\end{equation}
Let us denote by $\Pi_1$ the positive one-dimensional rank projection on $\Ker(\lambda^{\ep,1}_1I -S^\ep_1)$. 
Consider $\Co=\Co^\ep$ a closed circle with center $\lambda^{\ep,1}_1$ and the radius $\eta_1(\ep)$ given by
$$\eta_1(\ep)=\frac{1}{2}\left|\lambda^{\ep,1}_1-\lambda^{\ep,2}_1\right|,$$
so that the resolvent $(\lambda I -S^\ep_1)^{-1}$ exists for every $\lambda\in \Co $.
	Recalling the formula for  spectral projectors \cite[Theorem 1.5.4]{Davies2007}, we obtain for $\ep$ sufficiently small:
\begin{equation*}
\begin{aligned}
\Theta_1 A^\ep-\Pi_1 (\Theta_1 A^\ep)&=\displaystyle \dfrac{1}{2i \pi}\oint_{\Co }(\lambda -\lambda^{\ep,1}_1 )^{-1}\dd\lambda \Theta_1  A^\ep-\dfrac{1}{2 i \pi}\oint_{\Co }(\lambda -S^\ep_1)^{-1}\dd\lambda \Theta_1 A^\ep \vspace{0.1cm} \\
&=\displaystyle \dfrac{1}{2i\pi}\oint_{\Co }(\lambda -S^\ep_1)^{-1}(\lambda -\lambda^{\ep,1}_1 )^{-1}(S^\ep_1-\lambda^{\ep,1}_1 )\Theta_1 A^\ep \dd\lambda.
\end{aligned}
\end{equation*}
As a consequence, since $S_1^\ep$ is self-adjoint, we obtain the following estimate:
\begin{equation*}
\begin{aligned}
\|\Theta_1 A^\ep-\Pi_1 (\Theta_1 A^\ep)\|_{L^2(\Omega_1)}&\leq \left(\dfrac{1}{\eta_1(\ep)}\right)^2 \|(\lambda^{\ep,1}_1 -S^\ep_1)(\Theta_1 A^\ep)\|_{L^2(\Omega_1)}\\
&\leq \left(\dfrac{2}{\left|\lambda^{\ep,1}_1-\lambda^{\ep,2}_1\right|}\right)^2 \|(\lambda^{\ep,1}_1 -S^\ep_1)\Theta_1 A^\ep\|_{L^2(\Omega_1)}.
\end{aligned}
\end{equation*}
Now recall that the spectral gap $\lambda^{\ep,1}_1-\lambda^{\ep,2}_1$ is at most polynomial (see Assumption \ref{Assump:gap}), so that \eqref{Eq:S-Lambda} leads us to the following estimate 
\begin{equation}\label{pq}
\|\Theta_1 A^\ep-\Pi_1 (\Theta_1 A^\ep)\|_{L^2(\Omega_1)}=o(\ep^\infty),\; \ep\ll 1.
\end{equation}
We remind that $(\lambda^{\ep,1}_1, \phi^{\ep,1}_1)$ is the principal eigenpair of $L^\ep_1$. Hence $(\lambda^{\ep,1}_1, \Theta_1 \phi^{\ep,1}_1)$ becomes the principal eigenpair of $S^\ep_1$ and the spectral projector $\Pi_1$ is given by
$$
\Pi_1(\varphi)=\|\Theta_1\phi_1^{\ep,1}\|_{L^2(\Omega_1)}^{-2}\Theta_1 \phi_1^{\ep,1}\left(\int_{\Omega_1} \Theta_1(x)\phi_1^{\ep,1}(x)\varphi(x)\dd x\right).
$$
Since $\Theta_1=0$ on $\Sigma_1\setminus\Omega_1$, \eqref{pq} becomes
\begin{equation}\label{Eq:A_kTheta_k}
\|\Theta_1 A^\ep-\alpha^\ep_1 \nu^\ep_1 \Theta_1 \phi^{\ep,1}_1\|_{L^2(\Omega_1)}=o(\ep^\infty)
\end{equation}
{where $\alpha^\ep_1>0$ is the constant defined by
\begin{equation}\label{Eq:Alphak}
\alpha^\ep_1:=\left(\dfrac{\|\Theta_1 \phi^{\ep,1}_1\|^{-2}_{L^2(\Omega_1)}}{\nu^\ep_1}\right)\int_{\Omega_1}\Theta_1^2(x)\phi^{\ep,1}_1(x)A^\ep(x)\dd x>0
\end{equation}
and will be investigated below.} Note now that since $A^\ep$ is uniformly bounded in $L^1(\R^N)$, then \eqref{Eq:A} together with Lemma \ref{Lemma:Decay} $(b)$ yield 
	$$\chi_{\Omega_1}\left(\theta+\int_{\R^N}\beta_1(z)A^\ep(z)\dd z\right)A^\ep=\Lambda \xi_1 \chi_{\Omega_1}(m_\ep \star (\Psi_1 A^\ep))+o(\ep^\infty),\;\forall \ep\ll 1.$$
Next we deduce from the above equality that, for $\ep$ sufficiently small,
\begin{multline*}
\left\|\left(\theta+\int_{\R^N}\beta_1(z)A^\ep(z)\dd z\right)A^\ep-\Lambda \xi_1 \left(m_\ep \star (\Psi_1 \alpha^\ep_1 \nu^\ep_1 \phi^{\ep,1}_1)\right)\right\|_{L^2(\Omega_1)} \\
	\begin{aligned}
		&\leq \left\|\Lambda \xi_1 \left(m_\ep \star (\Psi_1 A^\ep-\Psi_1 \alpha^\ep_1 \nu^\ep_1 \phi^{\ep,1}_1)\right)\right\|_{L^2(\Omega_1)}+o(\ep^\infty) \\
		&\leq \Lambda \xi_1  \|m_\ep\|_{L^1(\R^N)} \|\Theta_1\|_{L^\infty}\left\|\Theta_1 A^\ep-\alpha^\ep_1 \nu^\ep_1 \Theta_1 \phi^{\ep,1}_1\right\|^2_{L^2(\Omega_1)}+o(\ep^\infty),
	\end{aligned}
\end{multline*}
so that \eqref{Eq:A_kTheta_k} implies that
\begin{equation}\label{Eq:A_k-Beta_k}
\left\|\left(\theta+\int_{\R^N}\beta_1(z)A^\ep(z)\dd z\right)A^\ep-\Lambda \xi_1 \left(m_\ep \star (\Psi_1 \alpha^\ep_1 \nu^\ep_1 \phi^{\ep,1}_1)\right)\right\|_{L^2(\Omega_1)}=o(\ep^\infty).
\end{equation}
The above equality also rewrites as follows
\begin{multline}\label{Eq:A_k-Beta_k-bis}
\left\|\left(\theta+\int_{\R^N}\beta_1(z)A^\ep(z)\dd z\right)A^\ep-\alpha_1^\ep\nu_1^\ep \theta L_1^\ep (\phi^{\ep,1}_1)\right\|_{L^2(\Omega_1)}\\
=\left\|\left(\theta+\int_{\R^N}\beta_1(z)A^\ep(z)\dd z\right)A^\ep-\alpha_1^\ep\nu_1^\ep \theta \lambda_1^{\ep,1}{\phi^{\ep,1}_1}\right\|_{L^2(\Omega_1)}=o(\ep^\infty).
\end{multline}
On the other hand we deduce from \eqref{Eq:Def_Mu} and \eqref{Eq:Mu} that
\begin{equation}\label{Eq:beta_k-A_ep}
\theta \lambda^{\ep,1}_1=\theta+\int_{\R^N}\beta_1(z)A^\ep(z)\dd z+o(\ep^\infty),
\end{equation}
so that \eqref{Eq:A_k-Beta_k-bis} becomes
\begin{equation}\label{Eq:esti_lambda}
\left\|\theta \lambda^{\ep,1}_1 A^\ep-\theta \lambda^{\ep,1}_1 \alpha^\ep_1 \nu^\ep_1 {\phi^\ep_1}\right\|_{L^2(\Omega_1)}=o(\ep^\infty).
\end{equation}
Since $\lambda^{\ep,1}_1\to R_{0,1}>1$ as $\ep\to 0$, {then 
$$\lambda^{\ep,1}_1\geq \dfrac{R_{0,1}}{2}, \quad \forall \ep\ll 1.$$
It follows from \eqref{Eq:esti_lambda} that
$$\dfrac{\theta R_{0,1}}{2}\|A^\ep-\alpha^\ep_1 \nu^\ep_1 {\phi^{\ep,1}_1}\|_{L^2(\Omega_1)}\leq 
\left\|\theta \lambda^{\ep,1}_1 A^\ep-\theta \lambda^{\ep,1}_1 \alpha^\ep_1 \nu^\ep_1 {\phi^\ep_1}\right\|_{L^2(\Omega_1)}=
o(\ep^\infty),\;\forall \ep\ll 1.$$
whence
\begin{equation}\label{Eq:Omega_k}
\|A^\ep-\alpha^\ep_1 \nu^\ep_1 {\phi^{\ep,1}_1}\|_{L^2(\Omega_1)}=
o(\ep^\infty),\;\forall \ep\ll 1.
\end{equation}
}

To complete the proof of the lemma, it remains to show that $\alpha_1^\ep$ is close to $1$ when $\ep\to 0$. In the following we check that $\alpha_1^\ep=1+o(\ep^\infty)$ as $\ep\to 0$. {To do so, we first see that \eqref{nuk} rewrites as
\begin{equation}\label{Eq:Alpha_1}
-\alpha^\ep_1{\theta(\lambda_1^{\ep, 1}-1)}=-\alpha^\ep_1\int_{\mathbb R^N}\beta_1(z)\nu_1^\ep\phi_1^{\ep, 1}(z)\dd z
\end{equation}
and \eqref{Eq:beta_k-A_ep} can be rewritten as:
\begin{equation}\label{Eq:Alpha_2}
\theta(\lambda^{\ep,1}_1-1)=\int_{\R^N}\beta_1(z)A^\ep(z)\dd z+o(\ep^\infty).
\end{equation}
Summing \eqref{Eq:Alpha_1} and \eqref{Eq:Alpha_2}, we  get:
$$\theta(1-\alpha^\ep_1) (\lambda^{\ep,1}_1-1)=\int_{\R^N}\beta_1(z)\left(A^\ep(z)-\alpha^\ep_1 \nu^\ep_1  \phi^{\ep,1}_1(z)\right)\dd z+o(\ep^\infty).$$
Using \eqref{Eq:Omega_k} and since $\lambda^{\ep,1}_1-1\to R_{0,1}-1>0$ as $\ep\to 0$, the latter equation leads to}
$$1-\alpha^\ep_1=o(\ep^\infty)\text{ as $\ep\to 0$}$$
which completes the proof of the Lemma.
\end{proof}

Equipped with the above lemmas we are now in the position to complete the proof of Theorem~\ref{Thm:shape}.

\begin{proof}[Proof of Theorem \ref{Thm:shape}]
We split our argument into two parts.
We first consider the case where $R_{0,1}>1$ and $R_{0,2}>1$ and show that the result directly follows from Lemma \ref{LE3}. In a second step we investigate the case where $R_{0,1}>1$ and $R_{0,2}\leq 1$. Using the symmetry of the problem with respect to the indices, this covers all the possible cases. \medskip

\noindent {\bf First case:} We suppose that $R_{0,1}>1$ and $R_{0,2}>1$.
In this case, Lemma \ref{LE3} applies and ensures that
\begin{equation}\label{Eq:Conv}
\|{A^\ep}_{\mid \Omega_k}-\nu^\ep_k {\phi^{\ep,1}_k}_{\mid \Omega_k}\|_{L^2(\Omega_k)}=o(\ep^\infty),\;\forall \ep\ll 1
\end{equation}
for each $k\in\{1,2\}$. Moreover, since $A^\ep$ is a fixed point of $T^\ep$, we have
\begin{equation}\label{Eq:Aep}
A^\ep(x)=\Lambda \int_{\R^N}m_\ep(x-y)\left(\dfrac{\xi_1 \Psi_1(y)}{\theta+\int_{\mathbb R^N}\beta_1(s)A^\ep(s)\dd s}+\dfrac{\xi_2 \Psi_2(y)}{\theta+\int_{\mathbb R^N}\beta_2(s)A^\ep(s)\dd s}\right)A^\ep(y)\dd y
\end{equation}
for every $x\in \R^N$. It follows from \eqref{ei} that
$$\nu^\ep_k\phi^{\ep,1}_k(x)=\dfrac{\Lambda}{\theta\lambda^{\ep,1}_k}\int_{\Omega_k}m_\ep(x-y)\xi_k\Psi_k(y)\nu^\ep_k\phi^{\ep,1}_k(y)\dd y$$
for each $k\in\{1,2\}$, where we recall that $\Omega_k=\{\Psi_k>0\}$ and $\Omega=\Omega_1\sqcup \Omega_2$. Injecting the latter equation into \eqref{Eq:Aep} leads to
\begin{flalign*}
&\int_{\R^N\setminus \Omega}\left|A^\ep-\nu_1^\ep \phi^{\ep,1}_1-\nu^\ep_2 \phi^{\ep,1}_2\right|(x)\dd x \\
&\leq \sum_{k=1,2}\dfrac{\Lambda}{\theta \lambda^{\ep,1}_k}\int_{\R^N\setminus \Omega} \int_{\R^N}m_\ep(x-y)\xi_k \Psi_k(y)\left|\left(\dfrac{\theta \lambda^{\ep,1}_kA^\ep(y)}{\theta+\int_{\mathbb R^N}\beta_k(s)A^\ep(s)\dd s}\right)-\nu^\ep\phi^{\ep,1}_k(y)\right|\dd y \dd x.
\end{flalign*}
We then infer from \eqref{Eq:beta_k-A_ep} that
\begin{flalign*}
\dfrac{\theta \lambda^{\ep,1}_k}{\theta+\int_{\mathbb R^N}\beta_k(s)A^\ep(s)\dd s}=1+o(\ep^\infty)\text{ for each $k\in\{1,2\}$}.
\end{flalign*}
Recalling that $\lambda^{\ep,1}_k\to R_{0,k}>1$ as $\ep\to 0$, and that the family $\{A^\ep\}_{\ep>0}$ is uniformly bounded in $L^1(\R^N)$ (see Theorem \ref{Thm:main}), one deduces that
$$
\int_{\R^N\setminus \Omega} \left|A^\ep-\nu^\ep_1 \phi^{\ep,1}_1- \nu^\ep_2 \phi^{\ep,1}_2\right|(x)\dd x\leq
\dfrac{\Lambda}{M\theta}\sum_{k=1,2}\|\Psi_k\|_{L^\infty}\|{A^\ep}_{\mid \Omega_k}-\nu^\ep_k {\phi^{\ep,1}_k}_{\mid \Omega_k}\|_{L^1(\Omega_k)}
+o(\ep^\infty)=o(\ep^\infty)
$$
for some constant $M>0$. Here we have used \eqref{Eq:Conv}. 
	
Finally, since $\Vert \chi_{\Omega_1}\phi^{\ep, 1}_2\Vert_{L^1(\R^N)}=o(\ep^\infty)$ and $\Vert \chi_{\Omega_2}\phi^{\ep, 1}_1\Vert_{L^1(\R^N)}=o(\ep^\infty)$, we obtain  
$$
\Vert A^\ep-(\nu^\ep_1\phi^{\ep, 1}_1+\nu^\ep_2\phi^{\ep, 1}_2)\Vert_{L^1(\mathbb R^N)}=o(\ep^\infty),
$$
that proves the result in the case where $R_{0,1}>1$ and $R_{0,2}>1$.

\medskip

\noindent{\bf Second case:} We assume now that $R_{0,1}>1$ and $R_{0,2}\leq 1$.
Note that Lemma \ref{LE3} applies and ensures that \eqref{Eq:Conv} holds for {$k=1$}.
From Lemma \ref{Lemma:Decay} (b) and \eqref{Eq:Aep}, we get
\begin{flalign}\label{Eq:A_ep-petito}
\displaystyle \int_{\Omega_2}A^\ep(x)\dd x&=\dfrac{\Lambda \xi_2}{\theta+\int_{\mathbb R^N}\beta_2(s)A^\ep(s)\dd s}\int_{\Omega_2}\int_{\Omega_2}m_{\ep}(x-y)\Psi_2(y)A^\ep(y)\dd y \dd x+o(\ep^\infty)\\
&\leq \dfrac{\theta R_{0,2}\int_{\Omega_2}A^\ep(y)\dd y}{\theta+\int_{\mathbb R^N}\beta_2(s)A^\ep(s)\dd s}+o(\ep^\infty). \nonumber
\end{flalign}
It follows that
\begin{equation}\label{Eq:A_l}
\left(1-\dfrac{\theta R_{0,2}}{\theta+\|\beta_2 A^\ep\|_{L^1(\mathbb R^N)}}\right)\int_{\Omega_2} A^\ep(x)\dd x=o(\ep^\infty).
\end{equation}
Now we prove that the following estimate holds
\begin{equation}\label{Eq:A_l2}
	\int_{\Omega_2}A^\ep(x)\dd x=o(\ep^\infty)
\end{equation}
When $R_{0,2}<1$, \eqref{Eq:A_l} implies:
	\begin{equation*}
		(1-R_{0, 2})\int_{\Omega_2}A^\ep(x)\dd x \leq \left(1-\dfrac{R_{0, 2}}{1+\theta^{-1}\int_{\R^N}\beta_2(z)A_2^\ep(z)\dd z}\right)\int_{\Omega_2}A^\ep(x) \dd x=o(\varepsilon^\infty)
	\end{equation*}
	hence \eqref{Eq:A_l2} holds. 

Now suppose that $R_{0,2}=1$. From \eqref{Eq:A_l}, we see that
	$$\dfrac{\left(\int_{\Omega_2}\beta_2(x) A^\ep(x)\dd x\right)^2}{\Vert \beta_2\Vert_{L^\infty}\left(\theta+M\|\beta_2\|_{L^\infty}\right)}\leq \dfrac{\int_{\mathbb R^N}\beta_2(x) A^\ep(x)\dd x\int_{\Omega_2}A^\ep(x)\dd x
}{\theta+\|\beta_2 A^\ep\|_{L^1(\mathbb R^N)}}=o(\ep^\infty)$$
for some constant $M>0$ such that $\|A^\ep\|_{L^1(\R^N)}\leq M$ for all $\ep$ small. Therefore, we have  
\begin{equation}\label{Eq:beta}
\int_{\Omega_2}\beta_2(x)A^\ep(x)\dd x=o(\ep^\infty).
\end{equation}
Next \eqref{Eq:A_ep-petito} allows us to  control the quantity $\int_{\Omega_2}A^\ep(z)\dd z$ by $\int_{\Omega_2}\beta_2(z)A^\ep(z)\dd z$ as follows
\begin{flalign*}
\int_{\Omega_2}A^\ep(x)\dd x&\leq \dfrac{\Lambda \xi_2 \|r_2\|_{L^\infty}}{\delta \theta} \int_{\Omega_2}\beta_2(x)A^\ep(x)\dd x+o(\ep^\infty)
\end{flalign*}
and therefore \eqref{Eq:A_l2} holds. \medskip

To complete the proof of the theorem, it remains to show that  
$$\int_{\R^N\setminus \Omega} \left|A^\ep(z)-\nu^\ep_1 \phi^{\ep,1}_1(z)\right|\dd z= o(\ep^\infty).$$
To this end, we follow the proof of the first case to obtain
\begin{flalign*}
&\int_{\R^N\setminus \Omega}\left|A^\ep-\nu_1^\ep \phi^{\ep,1}_1\right|(x)\dd x\leq\dfrac{\Lambda \xi_2}{\theta+\int_{\mathbb R^N}\beta_2(x)A^\ep(x)dx}\|\Psi_2\|_{L^\infty}\int_{\Omega_2}A^\ep(x)dx \\
	&\qquad + \dfrac{\Lambda}{\theta \lambda^{\ep,1}_1}\int_{\R^N\setminus \Omega} \int_{\R^N}m_\ep(x-y)\xi_1 \Psi_1(y)\left|\left(\dfrac{\theta \lambda^{\ep,1}_1A^\ep(y)}{\theta+\int_{\mathbb R^N}\beta_1(s)A^\ep(s)\dd s}\right)-\nu^\ep\phi^{\ep,1}_1(y)\right|\dd y \dd x.
\end{flalign*}
From \eqref{Eq:A_l2}, we deduce that
$$
\int_{\R^N\setminus \Omega} \left|A^\ep-\nu^\ep_1 \phi^{\ep,1}_1\right|(x)\dd x\leq
\dfrac{\Lambda}{\theta}\|\Psi_1\|_{L^\infty}\|A^\ep-\nu^\ep_1 \phi^{\ep,1}_1\|_{L^1(\Omega_1)}
+o(\ep^\infty)=o(\ep^\infty),
$$
by using the fact that \eqref{Eq:Conv} holds for {$k=1$}, this concludes the proof of this second case and thus the proof of Theorem \ref{Thm:shape}.
\end{proof}

\section{Proof of Theorem \ref{Thm:uniqueness} } \label{Sec:Thm3}

In this section we handle the uniqueness of the endemic steady state for $\ep$ sufficiently small and we prove Theorem \ref{Thm:uniqueness}. To this end, we use degree theory (see \textit{e.g.} \cite{Brown2014, Zei-86}).  

Our strategy is as follows: we first derive estimates for the eigenvalues of the linearised equation around each stationary solution for all $\varepsilon>0$ small enough. In particular we show that  every positive stationary solution is locally stable for the discrete dynamical system generated by $T^\ep$.  Next, we compute the Leray-Schauder degree of the (nonlinear) operator in a subset of the positive cone which contains all the positive fixed {points}, and show that it is equal to one.  Because of the additivity property of the Leray-Schauder degree, these two arguments combined together show that there cannot be more than one stationary solution.

Recall that $T^\ep=T_1^\ep+T_2^\ep $ (see the definitions \eqref{T} and \eqref{eq:Tk}). In this section, in order to work in a solid cone of a Banach space, we will be mainly interested in some properties of $T^\ep$, $T_1^\ep$ and $T_2^\ep$ considered as operators acting on $\Co(\Sigma)$, $\Co(\Sigma_1)$ and $\Co(\Sigma_2)$, where, according to Assumption \ref{Assump:supports}, $\Sigma_1$ and $\Sigma_2$ are defined in \eqref{support} while $\Sigma$ denote the compact set given by
$$
\Sigma=\Sigma_1 \sqcup \Sigma_2.
$$ 
Recall also that $\Omega_k=\{\Psi_k>0\}$ and $\Omega=\Omega_1\sqcup \Omega_2$. And note that due to the definition of $\Psi_k$ in \eqref{Eq:Psi_k} one has $\Omega\subset\Sigma$ and $\Omega_k\subset \Sigma_k$ for each $k\in\{1,2\}$.

We will use the fact that the fixed-points of $T^\varepsilon$ are close to the fixed-point of the uncoupled problem
\begin{equation}\label{eq:uncoupled}
	A^{\ep, *}:=A_1^{\ep, *}+A_2^{\ep, *},
\end{equation}
where for each $k=1,2$, $A_k^{\ep, *} \in L^1(\R^N)\cap\Co_b(\R^N)$ is the unique nontrivial solution of $T_{k}^\ep A_k^{\ep, *}=A_k^{\ep, *}$ if $R_{0, k}>1$ and $A_k^{\ep, *}\equiv 0$ otherwise. 

	Recall finally that the spectra of $L_1^\ep$ and $L_2^\ep$, considered as bounded operators on  $L^p(\Omega_k)$, $L^p(\R^N)$ with $1\leq p<\infty$, or $\Co(\Sigma_k)$, consist in  a real sequence of decreasing eigenvalues, independent of the space considered (see Lemma \ref{Lemma:Prel_1}),  which we denote 
$$\sigma(L^\ep_k)=\{\lambda_k^{\ep, n}, n\geq 1\},\quad k=1,2.$$

\begin{lemma}[Computation of the spectrum]\label{Lemma:Spec}
	Assume that $R_{0,1}>1$ and   that  one of the following properties is satisfied: 
	\begin{itemize}
		\item either $R_{0,2}\neq1$, 
		\item or $R_{0,2}=1 $ and the convergence of $r_\sigma(L_2^\ep)$ is at most polynomial for small $\ep$, namely $r_\sigma(L_2^\ep)\leq 1-C\ep ^M$ for some constants $C>0$ and $M>0$.
	\end{itemize}
	Then, there exists $\ep_0>0$ such that for any $0<\ep\leq \ep_0$ and for any nonnegative nontrivial fixed point $A^{\ep}\in \Co(\Sigma)$ of $T^{\ep}$, we have
\begin{equation*}\label{Eq:Spec}
\sigma(D_{A^{\ep}} T^{\ep})\subset (-1,1),
\end{equation*}
	wherein $D_{A^\ep}T^\ep $ denotes the Fréchet derivative of $T^\ep$ with respect to the $\Co(\Sigma)-$topology.
\end{lemma}

\begin{proof}
	We divide the proof into three steps.
\medskip

	\noindent\textbf{Step one:} We show that
	\begin{equation*}
		\sigma(D_{A_k^{\ep,*}}T^{\ep}_{k})=\{0\}\cup \left\{\dfrac{\lambda^{\ep,n}_k}{\lambda^{\ep,1}_k}\right\}_{n\geq 2} \cup \left\{\dfrac{1}{\lambda^{\ep,1}_k}\right\}, \qquad \forall k \in\{1,2\},
	\end{equation*}
	if $R_{0,k}>1$, and 
	\begin{equation*}
		\sigma(D_{A_k^{\ep,*}}T^{\ep}_{k})=\sigma(L_k^\ep)=\{0\}\cup \left\{\lambda^{\ep,n}_k\right\}_{n\geq 1} 
	\end{equation*}
	otherwise, 
	where $A_k^{\ep, *} $ is the solution to the uncoupled problem $T^\ep_kA^{\ep, *}_k=A^{\ep, *}_k$ while $D_{A_k^{\ep,*}}T^{\ep}_{k}$ is the Fréchet differential of $T^\ep_k$ for the $\Co(\Sigma_k)$ topology.\medskip

	Let us consider the case $R_{0, k}>1$. We first recall that $T^\ep_k$ is compact as $0\le T^\ep_k \le L^\ep_k$ and as $L^\ep_k$ is compact by  Lemma \ref{Lemma:Prel_1}, then its Fréchet differential is also compact and its spectrum is consequently identical to its point spectrum.  Let $k\in\{1,2\}$ be given and let $L^2_{\Psi_k}(\Omega_k)$ be the weighted $L^2$ space defined by the inner product $\langle f,g\rangle_{{\Psi_k}}:=\int_{\Omega_k}f(z)g(z)\Psi_k(z)\dd z$. Since ${L_k^\ep}_{\mid L^2_{\Psi_k}(\Omega_k)}$ is self-adjoint in the space $L^2_{\Psi_k}(\Omega_k)$, there exists an Hilbert basis of $L^2_{\Psi_k}(\Omega_k)$ composed of eigenfunctions of the operator ${L_k^\ep}_{\mid L^2_{\Psi_k}(\Omega_k)}$, which we denote  $\{\phi^{\ep,n}_k\}_{n\geq 1}$, and related to the sequence of eigenvalues $\{\lambda^{\ep,n}_k\}_{n\geq 1}$. Observe that
	$$\forall f\in \Co(\Sigma_k) : {f}_{\mid \Omega_k} \in L^2_{\Psi_k}(\Omega_k)$$
	since $\Psi_k\in L^\infty(\R^N)$ and $\Sigma_k$ is compact. Observe also that, contrary to the previous sections, here $\phi^{\ep, 1}_k$ is not normalized in $L^1(\mathbb R^N)$ but in $L^2_{\Psi_k}(\Omega_k)$, namely $\Vert \phi^{\ep, 1}_k\Vert_{L^2_{\Psi_k}(\Omega_k)}=1$.

	Moreover, every $\phi^{\ep, n}_k$  can be extended to a function in $L^1(\mathbb R^N)\cap \Co_b(\mathbb R^N)$ by the identity:
$${\phi}^{\ep,n}_k(x):=\frac{1}{\lambda^{\ep,n}_k}\int_{\Omega_k}m_\ep(x-y)\Psi_k(y)\phi^{\ep,n}_k(y)\dd y , \quad x\in \mathbb R^N\setminus \Omega_k.$$
Let $h\in \Co(\Sigma_k)$ be given. Then  we have 
	\begin{equation*}
		D_{A_k^{\ep,*}}T^{\ep}_{k}h=\dfrac{{L_k^\ep} h}{1+\theta^{-1}\int_{\R^N}\beta_k(y)A_k^{\ep, *}(y)\dd y}-\dfrac{{L_k^\ep} A_k^{\ep, *}}{\left(1+\theta^{-1}\int_{\R^N}\beta_k(y)A_k^{\ep, *}(y)\dd y\right)^2} \frac{\int_{\R^N} \beta_k(y)h(y)\dd y}{\theta}.
	\end{equation*}
		 Recalling that $1+\theta^{-1}\int_{\mathbb R^N}\beta_k(z)A_k^{\ep, *}(z)\dd z=\lambda_k^{\ep, 1}$ and that $A^{\ep, *}_k=\theta\frac{\lambda_k^{\ep, 1}-1}{\int_{\mathbb R^N}\beta_k(z)\phi^{\ep, 1}_k(z)\dd z}\phi^{\ep, 1}_k$, {we note that $ D_{A_k^{\ep,*}}T^{\ep}_{k}h $ may also be expressed as 
		 \begin{equation*}
			 D_{A_k^{\ep,*}}T^{\ep}_{k}h=\dfrac{{L_k^\ep} h}{\lambda_k^{\varepsilon, 1}}-\dfrac{\lambda_k^{\varepsilon, 1}-1}{\lambda_k^{\varepsilon, 1}} \frac{\int_{\R^N} \beta_k(y)h(y)\dd y}{\int_{\mathbb R^N}\beta_k(z)\phi_k^{\varepsilon, 1}(z)\dd z}\phi_k^{\varepsilon, 1}.
		 \end{equation*}}
		 Let us write $h^n:=\langle h,  \phi^{\ep,n}_k\rangle_{{\Psi_k}}$, we compute
	\begin{align}
		\langle D_{A_k^{\ep,*}}T^{\ep}_{k}h, \phi^{\ep, n}_k\rangle_{\Psi_k}&=
		\begin{cases} 
			\dfrac{	h^1\lambda_k^{\ep, 1}-\dfrac{\lambda_k^{\ep, 1}\langle A^{\ep, *}_k, \phi^{\ep, 1}_k\rangle_{\Psi_k}}{1+\theta^{-1}\int_{\R^N}\beta_k(y)A_k^{\ep, *}(y)\dd y}   \displaystyle \int_{\R^N} \frac{\beta_k(y)}{\theta}h(y)\dd y}{1+\theta^{-1}\int_{\R^N}\beta_k(y)A_k^{\ep, *}(y)\dd y} , & \text{ if } n=1,\vspace{2pt}\\
			\dfrac{ h^n {\lambda^{\ep,n}_k}}{1+\theta^{-1}\int_{\R^N}\beta_k(y) A_k^{\ep, *}(y)\dd y}, & \text{otherwise},
		\end{cases}\nonumber \\
		&=
		\begin{cases} 
			h^1-\frac{\lambda_k^{\ep, 1}-1}{\lambda_k^{\ep, 1}}  \frac{\int_{\R^N} \beta_k(y)h(y)\dd y}{\int_{\mathbb R^N}\beta_k(z)\phi^{\ep, 1}_k(z)\dd z} , & \text{ if } n=1,\vspace{2pt}\\
			\dfrac{ {\lambda^{\ep,n}_k}}{\lambda_k^{\ep, 1}}h^n , & \text{otherwise}.
		\end{cases}\label{eq:eigfun-uncoupled}
	\end{align}
	We deduce that $\phi^{\ep, 1}_k$ is an eigenvector of $D_{A^{\ep, *}_k}T^\ep_k$ associated with the eigenvalue $\frac{1}{\lambda_k^{\ep, 1}}$, and  that every function 
	\begin{equation*}
		\tilde \phi^{\ep, n}_k:=\phi^{\ep, 1}_k+\left(\dfrac{1-\lambda^{\ep,n}_k}{\lambda_k^{\ep, 1}-1}\right)\dfrac{\int_{\mathbb R^N}\beta_k(z)\phi^{\ep, 1}_k(z)\dd z}{\int_{\mathbb R^N}\beta_k(z)\phi^{\ep, n}_k(z)\dd z}\phi^{\ep, n}_k
	\end{equation*}
	is an eigenvector of $D_{A^{\ep, *}_k}T^\ep_k$  associated with the eigenvalue $\frac{\lambda_k^{\ep, n}}{\lambda_k^{\ep, 1}}$. Thus:
\begin{align*} 
	\sigma(D_{A_k^{\ep,*}}T^{\ep}_{k})&\supset\{0\}\cup \left\{\dfrac{\lambda^{\ep,n}_k}{\lambda^{\ep,1}_k}\right\}_{n\geq 2} \cup \left\{\dfrac{1}{\lambda^{\ep,1}_k}\right\} 
\end{align*}

	Conversely let $\lambda\in\sigma(D_{A_k^{\ep,*}}T^{\ep}_{k})\setminus\{0\}$ be given and $h\in \Co(\Sigma_k)\setminus\{0\}$ be an associated eigenfunction. If $\supp h\subset \Sigma_k\setminus \Omega_k$ then {$ L_k^\varepsilon h=\frac{\Lambda\xi_k}{\theta} m_\varepsilon\star (\Psi_k h)=0$ because $\Psi_k$ is supported in $\Omega_k$, and therefore}
\begin{equation*}
	{D_{A_k^{\ep,*}}T^{\ep}_{k}h=}\lambda h=-\frac{\lambda_k^{\ep, 1}-1}{\lambda_k^{\ep, 1}}  \frac{\int_{\R^N} \beta_k(y)h(y)\dd y}{\int_{\mathbb R^N}\beta_k(z)\phi^{\ep, 1}_k(z)\dd z}\phi_k^{\ep, 1}, 
\end{equation*}
	which implies $h=\phi^{\ep,1}_k$ (up to the multiplication by a nonzero scalar), and this is a contradiction.  Therefore $\supp h\cap\Omega_k\neq\varnothing$. Then, taking the scalar product {of $ D_{A_k^{\ep,*}}T^{\ep}_{k}h $} with $\phi^{\ep, n}_k$ one finds that \eqref{eq:eigfun-uncoupled} still holds, {i.e., 
	\begin{equation*}
\lambda h^n =\langle D_{A_k^{\ep,*}}T^{\ep}_{k}h, \phi^{\ep, n}_k\rangle_{\Psi_k}=\begin{cases} 
			h^1-\frac{\lambda_k^{\ep, 1}-1}{\lambda_k^{\ep, 1}}  \frac{\int_{\R^N} \beta_k(y)h(y)\dd y}{\int_{\mathbb R^N}\beta_k(z)\phi^{\ep, 1}_k(z)\dd z} , & \text{ if } n=1,\vspace{2pt}\\
			\dfrac{ {\lambda^{\ep,n}_k}}{\lambda_k^{\ep, 1}}h^n , & \text{otherwise}.
		\end{cases}
	\end{equation*} }
	In particular, $\lambda $ is either one of the $\frac{\lambda^{\ep, n}_k}{\lambda^{\ep, 1}_k}$ {(if there is $n>1$ such that $h^n\neq 0$)} or $\frac{1}{\lambda^{\ep, 1}_k}$ {(if $h^n=0$ for all $n>1$)}. We have shown:
\begin{align*} 
	\sigma(D_{A_k^{\ep,*}}T^{\ep}_{k})&\subset\{0\}\cup \left\{\dfrac{\lambda^{\ep,n}_k}{\lambda^{\ep,1}_k}\right\}_{n\geq 2} \cup \left\{\dfrac{1}{\lambda^{\ep,1}_k}\right\} ,
\end{align*}
hence the equality holds.

If now $R_{0,1}\leq 1$, we have $A^{\ep, *}_k\equiv 0$ and therefore $D_{A^{\ep, *}_k}T^\ep_k=L_k^\ep$.  Then 
	\begin{equation*}
		\sigma(D_{A_k^{\ep,*}}T^{\ep}_{k})=\sigma(L_k^\ep)=\{0\}\cup \left\{\lambda^{\ep,n}_k\right\}_{n\geq 1} .
	\end{equation*}
Since $\lambda^{\ep,n}_k< \lambda^{\ep,1}_k$ for any $k\in\{1,2\}$ and $n\geq 2$, we deduce that whenever $R_{0,k}\neq 1$, there exists $\ep_0>0$ such that for every $\ep\in(0,\ep_0]$, we have:
\begin{equation}\label{Proof:Step1}
	\sigma\left(D_{A_k^{\ep,*}}T^{\ep}_{k}\right)\subset [0,1).
\end{equation}
If $R_{0,k}=1$, then \eqref{Proof:Step1} holds because of our assumption that $\lambda^{\ep, 1}_k\leq 1-C\ep^M$.

\medskip

\noindent\textbf{Step two:} For each $\ep>0$, let $\lambda^\ep\in \sigma(D_{A^{\ep}}T^{\ep})\setminus\{0\}$ be given and consider a bounded family of associated eigenvectors $h^\ep\in \Co(\Sigma)$. We prove that 
\begin{equation}\label{Proof:Step2}
	\begin{aligned}
		\sup_{\Sigma_k}\left|(D_{A_k^{\ep,*}}T^{\ep}_{k} -\lambda^\ep I)h_k^\ep\right|&=o(\varepsilon^\infty) , & k&=1,2,
	\end{aligned}
\end{equation}
for $\varepsilon>0$ sufficiently small, wherein we have set $h_1^\ep:=\chi_{\Sigma_1}h^\ep$ and $h_2^\ep:=\chi_{\Sigma_2}h^\ep$.\medskip

Let us show the property for $k=1$. The case $k=2$ is similar. We rewrite the identity
$\chi_{\Sigma_1}D_{A^{\ep}}T^{\ep} h^\ep=\lambda^\ep h_1^\ep$
as follows
\begin{equation}\label{eq:lemuniqueness-sys}
	(D_{A_1^{\ep,*}}T^{\ep}_{1} -\lambda^\ep I)h^\ep_1=(D_{A_1^{\ep,*}}T^{\ep}_{1}h^\ep_1-D_{A_1^{\ep}}T^{\ep}_{1}h^\ep_1)-D_{A_2^{\ep}}T^{\ep}_{2}h^\ep_2 \quad \text{in } \Sigma_1.
\end{equation}
Our next task is to show that the right-hand side of the previous equation has order $o(\varepsilon^\infty)$.
We first remark that, by Lemma \ref{Lemma:Decay}, we have
\begin{equation}\label{Proof:Step2_3}
	\sup_{x\in \Sigma_1}\left\vert D_{A_2^\ep}T_2^\ep h^\ep_2(x)\right\vert=o(\varepsilon^\infty). 
\end{equation}
Next we claim that, for $k\in\{1,2\}$, one has
\begin{equation}\label{Proof:Step2_4}
	\sup_{\Sigma_k} \left |D_{A_1^{\ep}}T^{\ep}_{1}h^\ep_1-D_{A_1^{\ep,*}}T^{\ep}_{1}h^\ep_1\right |=o(\ep^\infty).
\end{equation}
Indeed, we have
\begin{multline}\label{eq:lemuniqueness-op}
	D_{A_1^{\ep}}T^{\ep}_{1}h_1^\ep-D_{A_1^{\ep,*}}T^{\ep}_{1}h_1^\ep=\left(\frac{1}{1+\theta^{-1}\int_{\R^N}\beta_1(y)A_1^\ep(y)\dd y}-\frac{1}{1+\theta^{-1}\int_{\R^N}\beta_1(y)A_1^{\ep, *}(y)\dd y}\right) {L_1^\ep} h^\ep_1 \\
	 -\left(\frac{{L_1^\ep}A_1^\ep}{\left(1+\theta^{-1}\int_{\R^N}\beta_1(y)A_1^\ep(y)\dd y\right)^2}-\frac{{L_1^\ep} A_1^{\ep, *}}{\left(1+\theta^{-1}\int_{\R^N}\beta_1(y)A_1^{\ep, *}(y)\dd y\right)^2}\right)\int_{\R^N}\frac{\beta_1(y)}{\theta}h^\ep_1(y)\dd y.
\end{multline}
On the one hand, using Theorem \ref{Thm:shape}, we have
\begin{equation*}
	\left|\dfrac{1}{1+\theta^{-1}\int_{\R^N}\beta_1(y)A^{\ep}_1(y)\dd y}-\dfrac{1}{1+\theta^{-1}\int_{\R^N}\beta_1(y)A^{\ep,*}_1(y)\dd y}\right|\leq \dfrac{\|\beta_1\|_{L^\infty}}{\theta^2}\|A^{\ep}_1-A^{\ep,*}_1\|_{L^1\left(\Sigma_1\right)}=o(\ep^\infty), 
\end{equation*}
which settles the first term on the right-hand side of \eqref{eq:lemuniqueness-op}. On the other hand, we also have
\begin{multline*}
	\frac{{L_1^\ep} A_1^\ep}{\left(1+\theta^{-1}\int_{\R^N}\beta_1(y)A_1^\ep(y)\dd y\right)^2}-\frac{{L_1^\ep} A_1^{\ep, *}}{\left(1+\theta^{-1}\int_{\R^N}\beta_1(y)A_1^{\ep, *}(y)\dd y\right)^2}= \frac{{L_1^\ep} (A_1^\ep-A_1^{\ep, *})}{\left(1+\theta^{-1}\int_{\R^N}\beta_1(y)A_1^\ep(y)\dd y\right)^2}\\
	+\left(\frac{1}{\left(1+\theta^{-1}\int_{\R^N}\beta_1(y)A_1^{\ep}(y)\dd y\right)^2}-\frac{1}{\left(1+\theta^{-1}\int_{\R^N}\beta_1(y)A_1^{\ep, *}(y)\dd y\right)^2}\right){L_1^\ep} A_1^{\ep, *},
\end{multline*}
and, for all $x\in\Sigma_1$
\begin{flalign*}
	\left|{L_1^\ep}(A_1^\ep-A_1^{\ep, *})\right|(x)&=
\dfrac{|\xi_1 \int_{\Omega_1}m_\ep(x-y)\Psi_1(y)(A^{\ep}_1(y)-A^{\ep,*}_1(y))\dd y|}{(\theta+\int_{\R^N}\beta_1(y)A^{\ep,*}_1(y)\dd y)^2}  \\
	&\leq \dfrac{\xi_1}{\theta^2}\|\Psi_1\|_{L^\infty} \int_{\Omega_1}m\left(\dfrac{x-y}{\ep}\right)\dfrac{|A^{\ep}_1(y)-A^{\ep,*}_1(y)|}{\ep^N}\dd y \\
&\leq \dfrac{\xi_1}{\theta^2}\|\Psi_1\|_{L^\infty}\|m\|_{L^\infty}\dfrac{\|A^{\ep}_1-A^{\ep,*}_1\|_{L^1(\Omega_1)}}{\ep^N}=o(\ep^\infty),
\end{flalign*}
thus \eqref{Proof:Step2_4} holds. Combining \eqref{eq:lemuniqueness-sys}, \eqref{Proof:Step2_3} and \eqref{eq:lemuniqueness-op}, we have indeed shown \eqref{Proof:Step2}.

\medskip

\noindent\textbf{Step three:} Assume by contradiction that there exists a sequence $\lambda^\ep\in\sigma(D_{A^\ep}T^\ep)$  with $\ep\to 0$ and such that 
\begin{equation*}
	|\lambda^\ep|\geq 1.
\end{equation*}
Let $h^\ep\in\Co(\Sigma)$ be a sequence of associated normed (in $\Co(\Sigma)$) eigenvectors. Then there is $k\in \{1,2\}$ such that $\sup_{\Sigma_k}|h_k^\ep|=1$ for infinitely many $\ep>0$. 
Using the symmetry with respect to the indices and the possible extraction of subsequences, we will assume in this step that $k=1$. 

Let us first consider the case where $R_{0, 1}>1$. Then, let us define $g^\ep:=(D_{A_1^{\ep, *}}T^\ep_1-\lambda^\ep I)h^\ep_1$. 

Due to \eqref{Proof:Step2} we have $\|g^\ep\|_{\Co(\Sigma_1)}=o(\varepsilon^\infty)$. Next taking the inner product with $\phi^{\ep, n}_1$ yields, as in \eqref{eq:eigfun-uncoupled},
\begin{equation*}
	\langle h^\ep_1, \phi^{\ep, n}_1\rangle_{\Psi_1}=\dfrac{1}{\frac{\lambda^{\varepsilon, n}_1}{\lambda^{\varepsilon, 1}_1}-\lambda^\ep}g^\ep_n, \quad \forall n\geq 2,
\end{equation*}
where $g^\ep_n:=\langle g^\ep, \phi^{\ep, n}_1\rangle_{\Psi_1}$. Then, 
\begin{align*}
	\left|\frac{\lambda^{\ep, n}_1}{\lambda^{\ep, 1}_1}-\lambda^\ep\right|\geq |\lambda^\ep|-\left|\frac{\lambda^{\ep, n}_1}{\lambda^{\ep, 1}_1}\right|\geq 1-\frac{|\lambda^{\ep, 2}_1|}{|\lambda^{\ep, 1}_1|}\geq \frac{|\lambda^{\ep, 1}_1|-|\lambda^{\ep, 2}_1|}{|\lambda^{\ep, 1}_1|}\geq C\ep^M, 
\end{align*}
for some $C>0$ and $M>0$ independent of $\varepsilon$ and $n$. This shows 
\begin{equation*}
	|\langle h^\ep_1, \phi^{\ep, n}_1\rangle_{\Psi_1}|=|g^\ep_n|\times \mathcal O(\ep^{-M}), \quad \forall n\geq 2
\end{equation*}
therefore
\begin{align*}
	\Vert h^\ep_1-\langle h^\ep_1, \phi^{\ep, 1}_1\rangle_{\Psi_1}\phi^{\ep, 1}_1\Vert_{L^2_{\Psi_1}}^2 &
	= \sum_{n=2}^{+\infty}|\langle h^\ep_1, \phi^{\ep, n}_1\rangle_{\Psi_1}|^2 
	= \sum_{n=2}^{+\infty}\left|\dfrac{1}{\frac{\lambda^{\varepsilon, n}_1}{\lambda^{\varepsilon, 1}_1}-\lambda^\ep}g_n^\ep\right|^2 \\
	& \leq C^{-2}\ep^{-2M} \sum_{n=2}^{+\infty} |g^\ep _n|^2 
	\leq C^{-2}\ep^{-2M}\Vert g^\ep \Vert_{L^2_{\Psi_1}}^2
	= o(\varepsilon^\infty)
\end{align*}
by using \eqref{Proof:Step2}. 

Set $\mu^\ep:=h^\ep_1-\langle h^\ep_1, \phi^{\ep, 1}_1\rangle_{\Psi_1}\phi^{\ep, 1}_1$, then we have $\Vert L_1^\ep \mu^\ep\Vert_{\Co(\Sigma_1)}=\mathcal O(\Vert \mu^\ep\Vert_{L^2_{\Psi_1}})=o(\varepsilon^\infty)$. By means of \eqref{Proof:Step2}, we deduce that
\begin{align}
	\lambda^\ep h^\ep_1+o(\ep^\infty)&=D_{A^{\ep, *}_1}T^\ep_1h^\ep_1=\frac{L^\ep_1(\mu^\ep+\langle h^\ep_1, \phi^{\ep, 1}_1\rangle_{\Psi_1}\phi^{\ep, 1}_1)}{\lambda^{\ep, 1}_1}-\frac{\lambda^{\ep, 1}_1-1}{\lambda^{\ep, 1}_1}\frac{\int_{\mathbb R^N}\beta_1(y)h_1^\ep(y)\dd y}{\int_{\mathbb R^N}\beta_1(y)\phi^{\ep, 1}_1(y)\dd y}\phi^{\ep, 1}_1 \nonumber \\
	&=\left(\langle h^\ep_1, \phi^{\ep, 1}_1\rangle_{\Psi_1}-\frac{\lambda^{\ep,1}_1-1}{\lambda^{\ep,1}_1} \frac{\int_{\mathbb R^N}\beta_1h^\ep_1}{\int_{\R^N}\beta_1\phi^{\ep,1}_1}\right)\phi^{\ep, 1}_1+o(\varepsilon^\infty):=\alpha^\ep\phi^{\ep, 1}_1+o(\ep^\infty). \label{Proof_Step3}
\end{align}
Next note that
\begin{equation*}
	1\leq |\lambda^\ep|\sup_{x\in \Sigma_1}|h^\ep_1(x)|=|\alpha^\ep|\sup_{x\in\Sigma_1}\phi^{\ep, 1}_1(x)+o(\ep^\infty),
\end{equation*}
where
\begin{align*}
	\sup_{x\in\Sigma_1}\phi^{\ep, 1}_1(x)&=\frac{1}{\lambda^{\ep, 1}_1}\sup_{x\in\Sigma_1} L_1^{\ep}\phi^{\ep, 1}_1(x)=\frac{1}{\lambda^{\ep,1}_1}\sup_{x\in\Sigma_1}\frac{\Lambda\xi_1}{\theta}\int_{\mathbb R^N}m_\ep (x-y)\Psi_1(y)\phi^{\ep, 1}_1(y)\dd y\\
	&\leq \frac{\Lambda\xi_1}{\lambda^{\ep, 1}_1\theta}\frac{\Vert m\Vert_{L^\infty}}{\ep^N} \Vert\Psi_1\Vert_{L^2(\mathbb R)}\Vert\phi^{\ep, 1}_1\Vert_{L^2_{\Psi_k}}=\mathcal O(\ep^{-N}), 
\end{align*}
therefore $|\alpha^\ep|\geq C\ep^N$ for some constant $C>0$. By definition of $h^\ep$ and using \eqref{Proof:Step2_3} {and} \eqref{Proof_Step3}, it follows that
\begin{equation*}
	o(\ep^\infty)=(D_{A^{\ep, 1}_1}T^\ep_1-\lambda^\ep I)h^\ep_1=\frac{1}{\lambda^\ep}(D_{A^{\ep, 1}_1}T^\ep_1-\lambda^\ep I)(\alpha^\ep\phi^{\ep, 1}_1+o(\ep^\infty))=\frac{\alpha^\ep}{\lambda^\ep}\left(\frac{1}{\lambda^{\ep, 1}_1}-\lambda^\ep\right)\phi^{\ep, 1}_1+o(\ep^\infty),
\end{equation*}
then multiplying by $\phi^{\ep, 1}_1\Psi_1$ and integrating, we get
\begin{equation*}
	\left|\frac{1}{\lambda^{\ep, 1}_1}-\lambda^\ep\right|=o(\varepsilon^\infty).
\end{equation*}
Since $\lambda^\ep\geq 1$ and $\lambda^{\ep, 1}_1\to R_{0, 1}>1$ as $\ep\to 0$, we obtain a contradiction.\medskip

Now we assume that $R_{0,1}\leq 1$, then we have $A_1^{\ep, *}\equiv 0$, hence
\begin{equation*}
	D_{A^{\ep, *}_1}T^\ep_1=L_1^\ep,
\end{equation*}
which leads us to 	
\begin{equation*}
	r_\sigma(D_{A_1^{\ep,*}}T^{\ep}_{1})=r_\sigma(L_1^\ep)\xrightarrow[\ep\to 0]{}R_{0, 1}\leq 1.
\end{equation*}
Moreover, by definition of $\lambda^\ep$ and using \eqref{Proof:Step2}, we have $(L^\ep_1-\lambda^\ep I)h^\ep_1=:g^\ep=o(\ep^\infty)$ hence 
\begin{equation*}
	\Vert h^\ep_1\Vert_{L^2_{\Psi_1}}\leq \Vert (L^\ep_1-\lambda^\ep I)^{-1}g^\ep\Vert_{L^2_{\Psi_1}}\leq \Vert (L^\ep_1-\lambda^\ep I)^{-1}\Vert_{\mathcal L(L^2_{\Psi_1})}\Vert g^\ep\Vert_{L^2_{\Psi_1}}=\frac{1}{\dist(\lambda^\ep, \sigma(L_1^\ep))}\Vert g^\ep\Vert_{L^2_{\Psi_1}}.
\end{equation*}
Now let us observe that there exists some constant $C>0$ such that $\Vert h^\ep_1\Vert_{L^2_{\Psi_1}}\geq C\ep^{N}$ for  $\ep$ sufficiently small. To see this, note that one has, for all $x\in \Sigma_1$,
\begin{align*}
	|h^\ep_1|(x)&=\frac{1}{|\lambda^\ep|}|L^\ep_1h^\ep_1(x)-g^\ep(x)|\leq \frac{1}{|\lambda^\ep|}\left(\frac{\Lambda\xi_1}{\theta}\int_{\mathbb R^N}m_\ep (x-y)|h^\ep_1|(y)\Psi_1(y)\dd y + |g^\ep|(x) \right) \\
	&\leq \frac{c}{\ep^N}\Vert h^\ep_1\Vert_{L^2_{\Psi_1}} + o(\ep^\infty),
\end{align*}
where $c>0$ is some constant independent of $\ep$. Finally recalling that $\|h_1^\ep\|_{\Co(\Sigma_1)}=1$ this proves the expected lower bound {$\Vert h^\ep_1\Vert_{L^2_{\Psi_1}}\geq C\ep^{N}$ for  $\ep$ sufficiently small.} This estimate allows us to conclude that 
\begin{equation*}
	\dist(\lambda^\ep, \sigma(L^\ep_1))=o(\ep^\infty), 
\end{equation*}
which is a contradiction since $\lambda^\ep\geq 1$ while 
\begin{equation*}
	\sup\{|\lambda|, \lambda\in\sigma(L^\ep_1)\}=r(L^\ep_1)\leq 1-C\ep^M,
\end{equation*}
by our assumptions and \eqref{Proof:Step1}. This completes the proof of Lemma \ref{Lemma:Spec}.
\end{proof}

Our next task is to compute the Leray-Schauder degree of the operator $T^\ep$ in a suitable subset of the positive cone, $\Co_+(\Sigma)$, of $\Co(\Sigma)$. For $\alpha>0$ we define the open set
\begin{equation*}
	K_{\alpha}:=\left\{A\in\Co(\Sigma):\, A(x)>\alpha  
	\quad\forall x\in \Sigma\right\}.
\end{equation*}
\begin{lemma}[Computation of the degree]\label{lem:LSD}
	Assume that $R_{0,1}>1$. Then, for  $\ep>0$ sufficiently small, there exists $\alpha=\alpha(\varepsilon)>0$ such that for any nonnegative nontrivial (thus positive) fixed point $A \in\Co(\Sigma)$ of $T^\ep$, we have:
	\begin{equation*}
		A\in K_{\alpha}.
	\end{equation*}
	Moreover, 
	\begin{equation}\label{eq:LSD}
		\deg\left( I-T^\ep, K_{\alpha}\right)=1,
	\end{equation}
	where $\deg$ denotes the Leray-Schauder degree. 
\end{lemma}
\begin{proof}
	Our proof relies on the construction of a suitable homotopy which allows us to separate the variables and compute the Leray-Schauder degree. For technical reasons, we do not use the same homotopy in the case $R_{0, 2}>1$ and $R_{0, 2}<1$. Therefore, we split the proof into two parts. \medskip

	\noindent{\sc Part 1: the case $R_{0,2}>1$.} Let us define, for $\tau\in[0,1]$, $A\in \Co_+(\Sigma)$ and $(A_1, A_2):=(\chi_{\Sigma_1}A, \chi_{\Sigma_2}A)$, the operators 
	\begin{equation}\label{eq:TT^ep_tau}
		\begin{aligned}
			T^{\ep, \tau}_1(A):&=\dfrac{\chi_{\Sigma_1}L_1^\varepsilon {A_1}}{1+\theta^{-1}\int_{\R^N}\beta_1(y)A_1(y)\dd y}+\tau\dfrac{\chi_{\Sigma_1}L_2^\varepsilon {A_2}}{1+\theta^{-1}\int_{\R^N}\beta_2(y)A_2(y)\dd y} =\chi_{\Sigma_1}T_{1}^\varepsilon A_1 + \tau\chi_{\Sigma_1} T_{2}^\ep A_2, \\
			T^{\ep, \tau}_2(A):&=\tau\dfrac{\chi_{\Sigma_2}L_1^\varepsilon {A_1}}{1+\theta^{-1}\int_{\R^N}\beta_1(y)A_1(y)\dd y}+\dfrac{\chi_{\Sigma_2}L_2^\varepsilon {A_2}}{1+\theta^{-1}\int_{\R^N}\beta_2(y)A_2(y)\dd y} =\tau \chi_{\Sigma_2}T_{1}^\ep A_1+\chi_{\Sigma_2} T_{2}^\varepsilon A_2 ,   \\
			T^{\ep, \tau}(A):&=T^{\ep, \tau}_1A+T^{\ep, \tau}_2 A
		\end{aligned}
		\end{equation}
		where $T^\ep_{k}$ is defined in \eqref{eq:Tk} for each $k\in\{1,2\}$. 
	The map $\left(\tau,A\right)\mapsto T^{\ep,\tau}(A)$ is continuous from $[0, 1]\times \Co_+(\Sigma)$ into $\Co(\Sigma)$. 
Let us first observe that there exists $M>0$ such that for all $\tau\in [0,1]$, if $A\in \Co_+(\Sigma)$ satisfies $A=T^{\ep,\tau}(A)$ then $\|A\|_{L^1(\Sigma)}\leq M$. One may also notice that this upper bound can be chosen independently of $\ep>0$.

We first show that the fixed points of $T^{\ep,\tau}$ can be estimated from below uniformly in $\tau\in (0,1)$. This will allow us to easily compute the Leray-Schauder degree, since $T^{\ep,0}$ is completely uncoupled in its variables.\medskip

	\noindent\textbf{Step 1:} We show that there exists $\varepsilon_0>0$ small enough such that for all $\varepsilon\in (0,\varepsilon_0)$ there exists $\alpha=\alpha(\varepsilon)>0$ such that for all $\tau\in [0,1]$ one has 
	\begin{equation}\label{eq:lowerbound_uniform_tau}
		\min_{x\in\Sigma}(A^\tau(x))>\alpha,
\end{equation}
	for any {$A^\tau$ satisfying $T^{\ep,\tau} A^\tau=A^\tau$ and $\min_{z\in\Sigma_1\cup \Sigma_2}A^\tau(z)>0$}.

{
First, since $\lim_{\varepsilon\to 0}r_\sigma(L_k^\varepsilon)=R_{0, k}>1$ ($k=1,2$), we can find $\ep_0>0$ such that $r_\sigma(L^\ep_1)>1$ and $r_\sigma(L^\ep_2)>1$ for any $\ep\in(0,\ep_0)$.\\
Let $\ep\in (0,\ep_0)$ be given and fixed.
Let $\tau\in[0,1]$ be given and $A^\tau\in \Co(\Sigma)$ be a fixed point of $T^{\ep,\tau}$ such that $\min_{z\in\Sigma_1\cup \Sigma_2}A^\tau(z)>0$.  
Set $(A^\tau_1, A^\tau_2):=(\chi_{\Sigma_1}A^\tau, \chi_{\Sigma_2}A^\tau)\in \Co_+(\Sigma_1)\times \Co_+(\Sigma_2)\subset L^1_+(\Sigma_1)\times L^1_+(\Sigma_2)$ so that $T^{\ep,\tau}(A^\tau_1+A^\tau_2)=(A^\tau_1+A^\tau_2)$ and $\min_{z\in \Sigma_1} A_1^\tau(z)>0$, $\min_{z\in \Sigma_2}A_2^\tau(z)>0$}.

	{Now reasoning as in the proof of the estimate \eqref{eq:lowerbound} in Lemma \ref{Lemma:Prel_3} we find that 
	\begin{equation}\label{eq:lowerbound-tau-1}
		\int_{\mathbb R^N}\beta_1(y)A^\tau(y)\dd y\geq \frac{\theta}{2}(r_\sigma(L_1^\ep)-1). 
	\end{equation}
	Indeed suppose by contradiction that $\int_{\mathbb R^N}\beta_1(y)A^\tau(y)\dd y<\frac{\theta}{2}(r_\sigma(L_1^\ep)-1) =:\eta$, then we also have $\int_{\mathbb R^N}\beta_1(y)\big[(T^{\ep, \tau})^n A^\tau\big](y)\dd y<\eta $ for any $n\geq 0$ and 
	\begin{align*}
		T^{\varepsilon, \tau}(A^\tau) & = T_1^{\varepsilon, \tau}(A^\tau)+T_2^{\varepsilon, \tau}(A^\tau)\geq T_k^{\varepsilon, \tau}(A^\tau)\geq \frac{\theta}{\theta+\eta}L_1^{\varepsilon} A^\tau \\ 
		(T^{\varepsilon, \tau})^2(A^\tau) & = T_1^{\varepsilon, \tau}(T^{\varepsilon, \tau} A^\tau) + T_2^{\varepsilon, \tau}(T^{\varepsilon, \tau} A^\tau) \geq T_1^{\varepsilon, \tau} (T^{\varepsilon, \tau} A^\tau)\geq \frac{\theta}{\theta+\eta} L_1^{\varepsilon} \left(T^{\varepsilon, \tau} A^\tau\right) \geq \left(\frac{\theta}{\theta+\eta} L_1^{\varepsilon}\right)^2A^\tau \\
		& \quad \vdots\\
		(T^{\varepsilon, \tau})^n A^\tau& = T_1^{\varepsilon, \tau} (T^{\varepsilon, \tau})^{n-1}A^\tau + T_2^{\varepsilon, \tau}(T^{\varepsilon, \tau})^{n-1}A^\tau\geq T_1^{\varepsilon, \tau}(T^{\varepsilon, \tau})^{n-1}A^\tau \geq \frac{\theta}{\theta+\eta} L_1^{\varepsilon} (T^{\varepsilon, \tau})^{n-1} A^\tau \\
		&\geq \left(\frac{\theta}{\theta+\eta} L_1^\varepsilon\right)^{n} A^\tau.
	\end{align*}
	Applying Lemma \ref{Lemma:Prel_2} leads to a contradiction and \eqref{eq:lowerbound-tau-1} follows. 

	Since $r_\sigma(L_2^\ep)>1 $, we similarly get  
	\begin{equation*}
		\int_{\mathbb R^N}\beta_2(y)A^\tau(y)\dd y\geq \frac{\theta}{2}(r_\sigma(L_2^\ep)-1).
	\end{equation*}}
	 We deduce that there exists $\eta>0$ (independent of $\ep\in (0,\ep_0)$, $\tau\in [0,1]$) such that for any positive fixed point $A^\tau$ of $T^{\ep,\tau}$ one has
\begin{equation}\label{Eq:eta}
\int_{\R^N}\beta_k(y)A^{\tau}_k(y)\dd y\geq \eta\text{ for any $k\in\{1,2\}$}.
\end{equation}
Next, using \eqref{small} and since the fixed points of $T^{\ep,\tau}$ are bounded by some constant $M$ in $L^1(\Sigma)$, we obtain $A^\tau_k=T^{\ep, \tau}_kA^\tau\leq L^\ep_kA^\tau_k+o(\ep^\infty)$ where $o(\ep^\infty)$ is uniform with respect to $\tau\in [0,1]$ and $x\in\Sigma_k$. Hence we get
	\begin{align*}
		\eta&\leq \int_{\mathbb R^N}\beta_k(x)A^\tau_k(x)\dd x\leq \frac{\Lambda\xi_k}{\theta}\iint_{\mathbb R^N\times \mathbb R^N}\beta_k(x)m_\ep (x-y)A^\tau_k(y)\Psi_k(y)\dd y\dd x +o(\ep^\infty)\\
		&\leq \frac{\Lambda\xi_k\Vert \beta_k\Vert_{L^\infty}}{\theta}\int_{\mathbb R^N}\Psi_k(y)A^\tau_k(y)\dd y+o(\ep^\infty).
	\end{align*}
	Thus we have for any $k\in\{1,2\}$ and any $x\in\Sigma_k$:
	\begin{align*}
		A^{\tau}_k(x)&{=\chi_{\Sigma_k}T^{\ep, \tau}A^\tau \geq T^{\ep}_k A^\tau }\geq \dfrac{\Lambda \xi_k}{\theta+M\|\beta_k\|_{L^\infty}}\int_{\Omega_k}m_\ep(x-y)\Psi_k(y)A^{\tau}_k(y)\dd y\\
		&\geq  \dfrac{\Lambda \xi_k}{\theta+M\|\beta_k\|_{L^\infty}} \min_{x\in \Sigma_k} \int_{\Omega_k}m_\ep(x-y)\Psi_k(y)A^{\tau}_k(y)\dd y\\
		&{\geq \dfrac{\Lambda \xi_k}{\theta+M\|\beta_k\|_{L^\infty}}  \int_{\Omega_k}\Psi_k(y)A^{\tau}_k(y)\dd y\min_{x,y\in \Sigma_k}m_\ep(x-y)}\geq c(\varepsilon)
	\end{align*}
	for some constants $M>0$ and $c(\ep)>0$ independent of $A^\tau$ and $\tau\in [0,1]$. This shows \eqref{eq:lowerbound_uniform_tau} and thus that, for $\ep>0$ sufficiently small, there exists $\alpha=\alpha(\ep)>0$ such that for any $\tau\in [0,1]$, any {positive} fixed points $A^{\tau}$ of $T^{\ep,\tau}$ satisfies $A^{\tau}\in K_{\alpha}$.\medskip

\noindent\textbf{Step 2:} We compute the Leray-Schauder degree of the operator $T^\ep$ in the open set $K_\alpha$. 

	We have shown in the previous step that $A\in K_\alpha$ for any positive fixed point of the operator $T^{\ep,\tau} $ with $\tau\in [0, 1]$. In particular, there is no fixed point of $T^{\ep,\tau}$ on the boundary of $K_\alpha$ for $\tau\in (0,1]$. For $\tau=0$, the operator $T^{\ep, 0}$ is uncoupled and hence we can compute the set of nonnegative fixed points of $T^{\ep,0}$, which is $\{(0, 0), (A^{\ep, *}_1, 0), (0, A^{\ep, *}_2),(A_1^{\ep, *}, A^{\ep, *}_2)\} $. None of those points lie in the boundary of $K_\alpha$. In particular, \cite[Theorem 11.8]{Brown2014} applies and shows that the Leray-Schauder degree in $K_\alpha$ is independent of $\tau$, \textit{i.e.}
\begin{equation*}
	\deg(I-T^{\ep, 0}, K_\alpha)=\deg(I-T^{\ep,1}, K_\alpha).
\end{equation*}
Since $T^{\ep, 0}$ is uncoupled with respect to $(A_1, A_2)\in \Co(\Sigma_1)\times\Co(\Sigma_2)$, the product property of the Leray-Schauder degree (see \cite[Theorem 11.3]{Brown2014}) implies that
\begin{equation*}
	\deg(I-T^{\ep, 0}, K_\alpha)=\deg(I-T^{\ep}_{1}, K^1_\alpha)\times\deg(I-T^\ep_{2}, K^2_{\alpha}), 
\end{equation*}
where $K^k_\alpha:=\{A_k\in \Co (\Sigma_k)\,|\, A_k(x)>\alpha, \forall x\in \Sigma_k\}$ for $k\in\{1,2\}$. Finally, since $T^\ep_{k} $  has exactly one fixed point in $K^k_\alpha$ and $1\notin \sigma\left(D_{A_k^{\ep, *}}T^\ep_{k}\right)$, the degree of the nonlinear operator $T^\ep_{k} $ can be linked to the degree of its Fréchet derivative near $A^{\ep, *}_k$ (see \cite[Theorem 22.3]{Brown2014})
\begin{equation*}
	\deg(I-T^\ep_{k}, K^k_\alpha)=\deg(I-D_{A_k^{\ep, *}}T^\ep_{k}, B(0, 1)),
\end{equation*}
where $B(0,1)$ is the open ball of radius $1$  in $\Co(\Sigma_k)$.
The explicit formula of the degree of linear operators (see \cite[Theorem 21.10]{Brown2014}) allows us to  conclude that
\begin{equation*}
	\deg(I-D_{A_k^{\ep, *}}T^\ep_{k}, B(0, 1))=1, 
\end{equation*}
since $\sigma(D_{A_k^{\ep, *}}T^\ep_{k})\subset (-1, 1)$ for $k\in\{1,2\}$. This shows \eqref{eq:LSD} and ends the proof of Lemma \ref{lem:LSD} in the case $R_{0, 2}>1$.\medskip

	\noindent{\sc Part 2: the case $R_{0,2}\leq 1$.} In this case we cannot use the same homotopy as in Part 1 to compute the Leray-Schauder degree, because $T^\ep_{2}$ has no nonnegative nontrivial fixed point. Instead, we define, for $\tau\in[0, 1]$, $A\in \Co_+(\Sigma)$ and $(A_1,A_2):=(\chi_{\Sigma_1} A, \chi_{\Sigma_2} A)$, the operators
	\begin{equation}\label{eq:TT^ep_tau-2}
		\begin{aligned}
			T^{\ep, \tau}_1(A) :&= \dfrac{\chi_{\Sigma_1}L_1^\varepsilon {A_1}}{1+\theta^{-1}\int_{\mathbb R^N}\beta_1(y)A_1(y)\dd y}+\dfrac{\chi_{\Sigma_1}L_2^{\varepsilon, \tau} {A_2}}{1+\theta^{-1}\int_{\mathbb R^N}\beta_2^\tau(y)A_2(y)\dd y}, \\
			T^{\ep, \tau}_2(A) :&= \dfrac{\chi_{\Sigma_2}L_1^\varepsilon {A_1}}{1+\theta^{-1}\int_{\mathbb R^N}\beta_1(y)A_1(y)\dd y}+\dfrac{\chi_{\Sigma_2} L_2^{\varepsilon, \tau} {A_2}}{1+\theta^{-1}\int_{\mathbb R^N}\beta_2^\tau(y)A_2(y)\dd y},\\
			T^{\ep, \tau}(A):&=T^{\ep, \tau}_1 A+T^{\ep, \tau}_2 A.
		\end{aligned}
	\end{equation}
	where $\beta_2^\tau(y):=\left(1+\tau\left(\frac{2}{R_{0,2}}-1\right)\right)\beta_2(y)$, $\Psi_2^\tau(y):=\frac{\beta_2^\tau(y)r_2(y)}{\delta(\theta + d_2(y))}$ and
	\begin{equation*}
L^{\ep, \tau}_2\varphi=\dfrac{\Lambda}{\theta}\int_{\Omega_2}m_{\ep }(x-y) \xi_2\Psi_2^\tau(y)\varphi(y)\dd y
	\end{equation*}
	which is well-defined since $R_{0,2}>0$ (recall  that $\Psi_2\not\equiv 0$ by Assumption \ref{Assump:Psi}). This corresponds to artificially increasing the basic reproductive number of the second equation until it becomes greater than 1. In particular, for $\tau =1$ we are in the same situation as in Part 1 since 
	$$\dfrac{\Lambda}{\theta}\xi_2 \|\Psi_2^1\|_{L^\infty}=2.$$
	Note that, as above, there exists $M>0$ such that for all $\tau\in [0,1]$, any fixed point $A^\tau\in \Co_+(\Sigma)$ of $T^{\ep, \tau}$ satisfies $\|A^\tau\|_{L^1(\Sigma)}\leq M$. 
Here our only task consists in finding a uniform lower bound for the fixed point of $T^{\ep,\tau}$. \medskip

	\noindent\textbf{Claim:} There is $\alpha>0$ such that for any $\tau\in(0,1)$ and any nonnegative nontrivial $A^\tau$ solution to $T^{\ep, \tau} A^\tau=A^\tau$, we have $A^\tau\in K_\alpha$.\medskip
	
	Indeed, let $A^\tau $ be such a fixed point. We first remark that, {as in Step one of Part one, reasoning as in the proof of the estimate \eqref{eq:lowerbound} in Lemma \ref{Lemma:Prel_3}  we get the estimate: 
}  
	\begin{equation*}
		\int_{\R^N}\beta_1(y)A^{\tau}_1(y)\dd y\geq \dfrac{\theta}{2}(r_\sigma(L^\ep_1)-1)>0.
	\end{equation*}
	Thus, we have
	\begin{equation*}
		A^{\tau}_1(x)\geq \dfrac{\Lambda \xi_k c(\ep)}{\theta+M \|\beta_1\|_{L^\infty} }\geq \eta>0, \qquad \forall x\in\Sigma_1
	\end{equation*}
	for some constants $M>0$ and $\eta>0$. To estimate $A_2^\tau$, we remark that
	$$A_2^\tau\geq \chi_{\Sigma_2}T^\ep_{1}A_1^\tau = \frac{\chi_{\Sigma_2}L_1^\ep {A_1^\tau}}{1+\theta^{-1}\int_{\R^N}\beta_1(y)A_1^\tau(y)\dd y}=\dfrac{\Lambda}{\theta}\dfrac{\chi_{\Sigma_2}\int_{\Omega_1}m_{\ep }(\cdot-y) \xi_1\Psi_1(y)A^\tau_1(y)\dd y}{1+\theta^{-1}\int_{\R^N}\beta_1(y)A_1^\tau(y)\dd y}$$
	and, as in Part one,  we have $A^\tau_1=T^{\ep, \tau}_1A^\tau\leq L^\ep_1A^\tau_1+o(\ep^\infty)$, and thus
	\begin{align*}
		\eta&\leq \int_{\mathbb R^N}\beta_1(x)A^\tau_1(x)\dd x\leq \frac{\Lambda\xi_1}{\theta}\iint_{\mathbb R^N\times \mathbb R^N}\beta_1(x)m_\ep (x-y)A^\tau_1(y)\Psi_1(y)\dd y\dd x +o(\ep^\infty)\\
		&\leq \frac{\Lambda\xi_1\Vert \beta_1\Vert_\infty}{\theta}\int_{\mathbb R^N}\Psi_1(y)A^\tau_1(y)\dd y+o(\ep^\infty).
	\end{align*}
	We conclude
	$$A^\tau_2(x) \geq \dfrac{\Lambda\xi_1 \eta}{\theta+M \|\beta_1\|_{L^\infty}}\min_{x\in \Sigma_2}\int_{\Omega_1}m_\ep(x-y)\dd y>0$$
	for every $x\in \Sigma_2$. This proves our Claim.\medskip

	To finish the proof of the second part, we remark that   the Leray-Schauder degree is independent of $\tau$ (see \cite[Theorem 11.8]{Brown2014}), \textit{i.e.}
	\begin{equation*}
		\deg(I-T^{\ep,0}, K_\alpha)=\deg(I-T^{\ep,1}, K_\alpha),
	\end{equation*}
	and we have proved in Part 1 that, for $\alpha$ sufficiently small, we have $ \deg(I-T^{\ep,1}, K_\alpha)=1$. This finishes the proof of Lemma \ref{lem:LSD}.
\end{proof}

\begin{lemma}\label{Lemma:Nbr_Eq}
	There exists $\ep_0>0$ such that for every $\ep\in(0,\ep_0]$, there is a finite number of nonnegative nontrivial fixed {points} of $T^{\ep}$. 
\end{lemma}

\begin{proof}
	Let $\ep>0$ and assume by contradiction that there exist infinitely many nontrivial (thus positive) fixed points of $T^{\ep}$. Since  $T^\ep$ is compact from $\Co_+(\Sigma)$ into itself, there exist a sequence $A^n\in \overline{K_\alpha}$ of fixed points of $T^\ep$ and    $\overline{A}$ such that
	$$\|A^n-\overline{A}\|_{L^\infty}\xrightarrow[n\to\infty]{}0.$$  
	By definition we have $T^{\ep}(A^n)=A^n$ for every $n\in \N$.  By the  continuity of $T^{\ep}$  we get 
	$$T^{\ep}(\overline{A})=\overline{A}.$$
	Since $T^\ep$ is Fréchet differentiable at the point $\overline A$, we have as $n\to\infty$
	$$\dfrac{\overline{A}-A^n}{\|\overline{A}-A^n\|_{L^\infty}}=\dfrac{T^{\ep}(\overline{A})-T^{\ep}(A^n)}{\|\overline{A}-A^n\|_{L^\infty}}=\dfrac{1}{\|\overline{A}-A^n\|_{L^\infty}}D_{\overline{A}}T^{\ep} \left(A^n-\overline{A}\right)+o(1).$$
	Let us define
	$$U^n:=\dfrac{\overline{A}-A^n}{\|\overline{A}-A^n\|_{L^\infty}},$$
	then we have
	$$U^n=D_{\overline{A}} T^{\ep} U^n+o(1)\text{ as }n\to\infty.$$
	By the compactness of $T^{\varepsilon}$, we can extract from $U^n$ a subsequence $\bar U^n$ which converges to $U^\infty$ with  $\Vert U^\infty\Vert_{L^\infty}=1$. We conclude
	$$U^\infty=D_{\overline{A}} T^{\ep,1}U^\infty$$
	which is a contradiction since $1\not\in \sigma(D_{\overline{A}} T^{\ep,1})$ by Lemma \ref{Lemma:Spec}. This finishes the proof of Lemma \ref{Lemma:Nbr_Eq}.
\end{proof}

We can finally prove our uniqueness result for $\varepsilon>0$ small.

\begin{proof}[Proof of Theorem \ref{Thm:uniqueness}.]
	By Lemma \ref{Lemma:Nbr_Eq}, there exists a finite number $N_\ep$ of fixed points of $T^\ep$. Denote by $A^{\ep,i}$, $i\in \llbracket 1,N_\ep \rrbracket$ an enumeration of the  fixed points of $T^\ep_\tau$. By the additivity property of the Leray-Schauder degree (see \cite[Theorem 11.4, p. 79]{Brown2014} and \cite[Theorem 11.5, p. 79]{Brown2014}), we get
	\begin{flalign}
		\deg\left(I-{T^{\ep}}, K_\alpha  \right)&=\deg\left(I-{T^{\ep}}, \bigcup_{i=1}^{N_\ep} B(A^{\ep,i}, \eta)\right) \nonumber \\
		&=\sum_{i=1}^{N_\ep} \deg\left(I-{T^{\ep}}, B(A^{\ep,i}, \eta)\right),
		\label{Eq:deg=N}
	\end{flalign}
	for $\eta>0$ sufficiently small, where $\alpha>0$ is the constant from Lemma \ref{lem:LSD} and $B(A^{\ep,i}, \eta)$ is the ball of center $A^{\ep, i}$ and of radius $\eta$ in $\Co(\Sigma)$. Next, using \cite[Theorem 22.3]{Brown2014}, we can link the degree of $T^\ep$ to the one of its Fréchet derivative close to a fixed point
	$$\deg\left(I-{T^{\ep,1}}, B(A^{\ep,i}, \eta)\right)=\deg\left(I-{D_{A^{\ep,i}} T^{\ep,1}}, B(0,1) \right)=1$$
	for $\eta>0$ sufficiently small and for every $i\in \llbracket 1,N_\ep \rrbracket$. This leads to
	$$\deg\left(I-{T^{\ep,1}}, K_\alpha \right)=N_\ep,$$
	where we have used  \eqref{Eq:deg=N}. Since we have shown in Lemma \ref{lem:LSD} that $\deg(I-T^\ep, K_\alpha)=1$, we conclude that $N_\ep=1$.  We have proven the uniqueness of the nonnegative nontrivial fixed point of $T^\ep$ for $\ep>0$ small, which completes the proof of Theorem \ref{Thm:uniqueness}.
\end{proof}

\section*{Acknowledgement}
The authors would like to thank the anonymous referee for his valuable comments which helped to improve the overall quality of the manuscript.

\appendix

\section{Spectral properties of a weighted convolution operator}\label{Appendix-spectral}

In this appendix, we state and recall some basic spectral properties of a weighted convolution operator as in \eqref{Eq:L_ep}, {\it i.e.} of the form
\begin{equation}\label{Eq:Appendix}
	L^\ep=m_\ep\star \left(\Psi\,\cdot\right),
\end{equation}
where $m_\ep=\ep^{-N} m\left(\ep^{-1}\cdot\right)$ with $\ep>0$.
Throughout this appendix, we assume
\begin{assumption}\label{ASS-Appendix}
	The function $m$ satisfies Assumption \ref{Assump:Psi} $c)$ and $\Psi:\R^N\to [0,\infty)$ is a non-zero continuous function tending to $0$ at $\|x\|=+\infty$.
\end{assumption}
The above assumption allows us to directly apply the results presented in this Appendix to operator $L^\ep$ as well as to $L_1^\ep$ and $L_2^\ep$ as defined in \eqref{Eq:Lk_ep}.

We start this section by reminding the following definition about positive operators:
\begin{definition}
	Let $p\in [1,\infty)$, $I\subset \R^N$ and $K\in\L(L^p(I))$ be given. We denote by
	$$L^p_+(I):=\{\varphi\in L^p(I): \varphi(x)\geq 0 \text{ a.e.}\}$$
	the positive cone of $L^p(I)$. Let $\langle \cdot, \cdot \rangle$ be the duality product between $L^p(I)$ and  $L^{p'}(I)$ where $1/p+1/p'=1$. For $\varphi \in L^p(I)$, the notation $\varphi\gneqq 0$ will refer to $\varphi\in L^p_+(I)$ and $\varphi\not\equiv 0$ while the notation $\varphi>0$ will refer to $\varphi\in L^p_+(I)$ and $\varphi(x)>0$ a.e. We say that
	\begin{enumerate}
		\item $K$ is positive if $K \left( L^p_+(I) \right) \subset L^p_+(I)$; 
		\item $K$ is said to be positivity improving if $K$ is positive and if, for every $\varphi\in L^p(I)$, $\varphi\gneqq 0$ and $\phi\in L^{p'}(I)$, $\phi\gneqq 0$,   we have $\langle K \varphi, \phi \rangle>0$.
	\end{enumerate}
\end{definition}
Consider the non-empty open set $\Omega\subset \R^N$ given by 
\begin{equation*}
	\Omega=\{x\in\R^N:\;{\Psi(x)}\in (0,\infty)\}.
\end{equation*}
We will denote, in the following lemma only, by $L^\ep_p, M^\ep_p$ , the operator $L^\ep$ defined in {\eqref{Eq:Appendix}} and considered as endomorphisms on $L^p(\Omega), L^p(\R^N)$ respectively. 
\begin{lemma}\label{Lemma:Prel_1}
	Let Assumption \ref{ASS-Appendix} be satisfied. Then the following properties are satisfied:
	\begin{enumerate}
		\item \label{item:compactness}
			Let $p\geq 1$. The operators $L^\ep_p$ and $M^\ep_p$ are compact, their spectra $\sigma(L^\ep_p)\backslash \{0\}$ and $\sigma(M^\ep_p)\backslash \{0\}$ are composed of isolated eigenvalues with finite algebraic multiplicity. All these operators share the same spectral radius -- independent of $p$ --  denoted by $r_\sigma(L^\ep)$, which is a positive algebraically simple eigenvalue. There exists a function $\phi^{\ep,1}_p\in L^p(\Omega)$ satisfying
			$$\phi^{\ep,1}_p> 0, \qquad L^{\ep}_p \phi^{\ep,1}_p=r_\sigma(L^{\ep})\phi^{\ep,1}_p.$$
			Moreover $L^\ep_p$ is positivity improving and, if $\phi\in L^p(\R^N)$, $ \phi\gneqq 0$ satisfies the equality $L^\ep_p \phi=\alpha \phi$ for some $\alpha \in \R$, then $\phi > 0$, $\phi\in \text{span}(\phi^{\ep,1}_p)$ and $\alpha=r_\sigma(L^\ep_p)$.  Finally, we have $\sigma(M^\ep_p)=\sigma(L^\ep_p)$.
		\item \label{item:symmetric-spectrum}
			Assume that $\Omega $ is bounded, let $S^\ep$ be the positive self-adjoint operator defined by
			\begin{equation}\label{Eq:S}
				S^{\ep}:L^2(\Omega)\ni \varphi(x) \mapsto \sqrt{\Psi(x)}\int_{\R^N}m_\ep(x-y)\sqrt{\Psi(y)}\varphi(y)\dd y\in L^2(\Omega),
			\end{equation}   
			then for every $p\geq 1$, we have $\sigma(S^\ep)=\sigma(L^\ep_p)\subset \R^+$,  
			and the following  Rayleigh formula holds
			\begin{equation}\label{Eq:Rayleigh}
				r_\sigma(L^\ep)=r_\sigma(S^\ep)=\underset{\|\varphi\|_{L^2(\Omega)}=1}{\sup_{\varphi\in L^2(\Omega)}}\int_{\Omega}\int_{\Omega} \sqrt{\Psi(x)}\sqrt{\Psi(y)}m_\ep(x-y)\varphi(x)\varphi(y)\dd x\dd y.
			\end{equation}

			Moreover,   $r_\sigma(L^\ep)$   satisfies
			\begin{align*}
				r_\sigma(L^\ep)&\xrightarrow[\ep \to 0]{}\|\Psi\|_{L^\infty}.
			\end{align*}

		\item Suppose that $\Omega$ bounded and let $\Sigma\supset \Omega$ be a compact set. The operator $L^\ep_\Sigma$, the realisation of $L^\ep$ in $\Co(\Sigma)$, is compact and one has $\sigma(L^\ep_\Sigma)=\sigma(L^\ep_p)$ for any $p\geq 1$.  
	\end{enumerate}
\end{lemma}
\begin{proof}
	Item \ref{item:compactness} is rather classical and has been proved in \cite[Theorem 4.1]{Djidjou2017}. In short, the inclusion $\sigma(M^\ep_p)\subset \sigma(L^\ep_p)$ is straightforward, while the reverse inclusion comes from the fact that any eigenfunction $\phi^\ep$ of $L_p^\ep$ related to the eigenvalue $\lambda^\ep$ can be extended from $L^p(\Omega)$ to $L^p(\R^N)$ by setting 
	\begin{equation}\label{Eq:Identity}
		\phi^{\ep}(x):=\frac{1}{\lambda^{\ep}}\int_{\Omega}m_\ep(x-y)\Psi(y)\phi^{\ep}(y)\dd y , \quad \forall x\in \mathbb R^N\setminus \Omega.
	\end{equation}  

	Let us show Item \ref{item:symmetric-spectrum}. Recall that $\Omega$ is bounded. Let $p\geq 1$, $\lambda\in\sigma(L^\ep_p)$ be an eigenvalue and $\varphi\in L^p(\Omega)\subset L^1(\Omega)$ be the associated eigenvector for $L^\ep_p$, \textit{i.e.}
	\begin{equation*}
		L^\ep_p\varphi(x)=\int_{\Omega}m_\ep(x-y)\Psi(y)\varphi(y)\dd y=\lambda \varphi(x)
	\end{equation*}
	so that $\varphi\in L^\infty(\Omega)$ by the Young inequality. Multiplying the above equation by $\sqrt{\Psi(x)}$, we get:
	\begin{equation*}
		\sqrt{\Psi(x)}\int_{\Omega}m_\ep(x-y)\sqrt{\Psi(y)}\sqrt{\Psi(y)}\varphi(y)\dd y=\lambda \sqrt{\Psi(x)}\varphi(x)\;
		\Longleftrightarrow\;  S^\ep\Phi(x) =\lambda \Phi(x), 
	\end{equation*}
	with $\Phi:=\sqrt{\Psi}\varphi\in L^\infty(\Omega)\subset L^2(\Omega)$. Therefore $\lambda\in\sigma(S^\ep)$. We have shown: 
	\begin{equation*}
		\sigma(L^\ep_p)\subset \sigma(S^\ep), \quad \forall p\geq 1.
	\end{equation*}
	Let us show the reverse inclusion. 
	Note that due to the first item, the operator $S^\ep$ is compact on $L^2(\Omega)$ and therefore  
	$\sigma(S^\ep)$ consists in isolated eigenvalues. Let $\lambda\in\sigma(S^\ep)\backslash\{0\}$ be an eigenvalue and $\Phi\in L^2(\Omega) $ be an associated eigenvector, so that
	\begin{equation*}
		\frac{\Phi(x)}{\sqrt{\Psi(x)}}=\frac{1}{\lambda}\cdot\int_{\Omega}m_\ep(x-y)\sqrt{\Psi(y)}\Phi(y)\dd y\in L^\infty(\Omega).
	\end{equation*}
	Hence there exists a non-zero function $\varphi\in L^\infty(\Omega)\subset L^p(\Omega), \forall p\geq 1$ such that $\Phi=\varphi\sqrt{\Psi}$ where the function $\varphi$ satisfies
	\begin{equation*}
		\lambda \varphi(x)=\int_{\Omega}m_\ep(x-y)\sqrt{\Psi(y)}\sqrt{\Psi(y)}\varphi(y)\dd y = L^\ep\varphi(x).
	\end{equation*}
	Thus $\lambda\in \sigma(L^\ep_p)$ for any $p\geq 1$ and we have shown 
	\begin{equation*}
		\sigma(S^\ep)\subset \sigma(L^\ep_p), \quad \forall p\geq 1.
	\end{equation*} 
	Formula \eqref{Eq:Rayleigh} is classical for positive and symmetric operators.

	Now let $\phi^{\ep, 1}$ be the positive eigenfunction of $L^\ep$ associated with $r_\sigma(L^\ep)$, normalised so that $\int_{\Omega}\phi^{\ep, 1}(y)\dd y=1$. We first notice that
	\begin{align*}
		r_\sigma(L^\ep)&=r_\sigma(L^\ep)\int_{\Omega}\phi^{\ep, 1}(x)\dd x= \iint_{\Omega\times \Omega}m_\ep(x-y)\Psi(y)\phi^{\ep,1}(y)\dd y\dd x\\
		&\leq \Vert \Psi\Vert_{L^\infty}\int_{\Omega}\int_{\Omega}m_\ep(x-y)\dd x\,\phi^{\ep,1}(y)\dd y\leq  \Vert \Psi\Vert_{L^\infty}.
	\end{align*}
	Next let $x_0\in\Omega$ be such that $\Psi(x_0)=\sup_{x\in\Omega}\Psi(x)$. Injecting the function $\frac{1}{\sqrt{|B(x_0, r)|}}\chi_{B(x_0, r)}(x)$ into \eqref{Eq:Rayleigh} yields
	\begin{align*}
		r_\sigma(L^\ep)=r_\sigma(S^\ep)&\geq \frac{1}{|B(x_0, r)|}\iint_{\Omega\times\Omega}\sqrt{\Psi(x)}\sqrt{\Psi(y)}m_\ep(x-y)\chi_{B(x_0, r)}(x)\chi_{B(x_0, r)}(y)\dd y \dd x\\
		&\geq\left(\inf_{x\in B(x_0, r)}\sqrt{\Psi(x)}\right)^2\frac{1}{|B(x_0, r)|}\iint_{B(x_0, r)^2}m_\ep(x-y)\dd x\dd y\\
		&=\inf_{x\in B(x_0, r)}\Psi(x)\frac{1}{|B(x_0, r)|}\varepsilon^N\iint_{B(x_0, r/\varepsilon)^2}m\left(x-y\right)\dd y\dd x\\
		&=\inf_{x\in B(x_0, r)}\Psi(x)\frac{\varepsilon^N\left|B\left(x_0, \frac{r}{\varepsilon}\right)\right|}{|B(x_0, r)|}\int_{B(0, r/\varepsilon)}m\left(y\right)\dd y\\
		&=\inf_{x\in B(x_0, r)}\Psi(x)\int_{B(0, r/\varepsilon)}m\left(y\right)\dd y\underset{\ep\to 0}{\longrightarrow}\inf_{x\in B(x_0, r)}\Psi(x),
	\end{align*}
	for all $r>0$ sufficiently small so that $B(x_0, r)\subset \Omega$. This proves the following inequality
	\begin{equation*}
		\inf_{x\in B(x_0, r)}\Psi(x)\leq \liminf_{\ep\to 0}r_\sigma(L^\ep)\leq \limsup_{\ep\to 0}r_\sigma(L^\ep)\leq \Vert\Psi\Vert_{L^\infty}.
	\end{equation*}
	Since $\lim_{r\to 0} \inf_{x\in B(x_0, r)}\Psi(x) = \Vert\Psi\Vert_{L^\infty}$, Item 2 is proved. 

	Finally we prove the last point, that is Item 3.  As $\Sigma$ is compact the fact that $L^\ep_\Sigma$ is compact follows from the Arzelà-Ascoli theorem. It remains to show that $\sigma(L^\ep_\Sigma)=\sigma(L^\ep_p)$ for any $p\geq 1$. The  inclusion $\sigma(L^\ep_\Sigma)\subset\sigma(L^\ep_p)$ is immediate since $\Co(\Sigma)\subset L^p(\Sigma)$ for every $p\geq 1$ and $\Omega\subset \Sigma$. Let $p\in [1,\infty)$ be given. The reverse inclusion follows from the identity \eqref{Eq:Identity} that allows to extend the eigenfunction from $L^p(\Omega)$ to $L^p(\Sigma)$. Let us notice  that $m_\ep \star (\Psi \phi) \in \Co(\Sigma)$ as soon as $\phi\in L^p(\Sigma)\subset L^1(\Sigma)$ (see \textit{e.g.} \cite[Corollary 3.9.6, p. 207]{Bogachev2007}).

	This ends the proof of Lemma \ref{Lemma:Prel_1}.
\end{proof}

We now give some asymptotic results for compact and positivity improving operators. The following result is classical but here we propose a proof for the sake of completeness.
\begin{lemma}\label{Lemma:Prel_2}
	Suppose that Assumption \ref{ASS-Appendix} holds and let $L^\ep$ be the operator defined in \eqref{Eq:Appendix}, considered as an operator from $L^1(\Omega)$ into itself.
	\begin{enumerate}
		\item \label{item:projector}
			The operator $L^\ep$ satisfies $r_\sigma(L^\ep)>0$ and
			$$\lim_{n\to \infty}\left\| \dfrac{(L^\ep)^n(\varphi)}{(r_\sigma(L^\ep))^n}-\Pi(\varphi)\right\|_{L^1(\Omega)}=0$$
			for every $\varphi\in L^1(\Omega)$, where $\Pi$ is the finite-rank projection into $\Ker\left(I-\frac{L^\ep}{r_\sigma(L^\ep)}\right)$. Moreover $\Pi$ is positivity improving. 
		\item \label{item:opnorm}
			If $r_\sigma(L^\ep)>1$, then 
			$$\lim_{n\to \infty}\|(L^\ep)^n(\varphi)\|_{L^1(\Omega)}=\infty$$
			for every $\varphi\in L^1_+(\Omega)\setminus \{0\}$. If $r_\sigma(L^\ep)<1$, then
			$$\lim_{n\to \infty}\|(L^\ep)^n(\varphi)\|_{L^1(\Omega)}=0$$
			for every $\varphi \in L^1_+(\Omega)\setminus\{0\}$.
	\end{enumerate}
\end{lemma}

\begin{proof}
	{\bf{Step one:}}  since $L^\ep$ is compact and positivity improving, then $r_\sigma(L^\ep)>0$ by \cite[Theorem 3]{Pagter86} and $r_\sigma(L^\ep)$ is a simple eigenvalue of $L^\ep$ (see Lemma \ref{Lemma:Prel_1}). We recall that 
	$$L^1(\Omega)=\Ker \left(I-\frac{L^\ep}{r_\sigma(L^\ep)}\right)\oplus \Rg(I-\Pi).$$  Moreover the projection $\Pi$ is given by the formula
	$$\Pi(\varphi)=\frac{ \langle \phi', \varphi\rangle}{\langle \phi', \phi\rangle}\phi$$
	where $\phi$ and $\phi'$ denote respectively the eigenfunctions of $L^\ep$ and its dual $(L^\ep)'$, associated to $r_\sigma(L^\ep)$. Note that $r_\sigma(L^\ep)$ is a pole of the resolvent of $L^\ep$ and an eigenvalue of $(L^\ep)'$ by the Krein-Rutman theorem (see \textit{e.g.}  \cite[Theorem 4.1.4, p. 250]{MeyerNieberg91} and \cite[Theorem 4.1.5, p. 251]{MeyerNieberg91}). Moreover, $\phi'\gg 0$ (see \textit{e.g.} \cite[Proposition 4]{Zerner87}). Consequently $\Pi$ is positivity improving and for every $\varphi\in L^1(\Omega)$, we have
	$${L^\ep}(\varphi)={L^\ep}(\Pi(\varphi))+{L^\ep}(I-\Pi)(\varphi)=r_\sigma({L^\ep})\Pi(\varphi)+{L^\ep}(I-\Pi)(\varphi).$$
	By induction, for every $n\geq 0$, we get
	\begin{equation*}
		({L^\ep})^n(\varphi)=(r_\sigma({L^\ep}))^n \Pi (\varphi)+\left[{L^\ep}(I-\Pi)\right]^n (\varphi).
	\end{equation*}
	Hence 
	\begin{align*}
		\left\|\dfrac{({L^\ep})^n(\varphi)}{(r_\sigma({L^\ep}))^n}-\Pi(\varphi)\right\|_{L^1({\Omega})}&=\dfrac{\left\|({L^\ep}(I-\Pi))^n (\varphi)\right\|_{L^1({\Omega})}}{(r_\sigma({L^\ep}))^n}\\
		&\leq \dfrac{\left\|({L^\ep}(I-\Pi))^n\right\|_{\L(L^1({\Omega}))}}{(r_\sigma({L^\ep}))^n}\|\varphi\|_{L^1({\Omega})}.
	\end{align*}
	On the one hand it is known (see \textit{e.g.} \cite[Theorem 1.5.4, p. 30]{Davies2007}) that
	$$\sigma({L^\ep}(I-\Pi))=\sigma({L^\ep})\setminus\{r_\sigma({L^\ep})\}, $$
	and therefore
	$$r_\sigma({L^\ep}(I-\Pi))<r_\sigma({L^\ep}).$$
	On the other hand, the Gelfand equality implies that
	$$r_\sigma({L^\ep}(I-\Pi))=\lim_{n\to \infty}\sqrt[n]{\|({L^\ep}(I-\Pi))^n\|_{\L(L^1({\Omega}))}}$$
	so that
	$$\|({L^\ep}(I-\Pi))^n\|_{\L(L^1({\Omega}))} \leq (r_\sigma({L^\ep}(I-\Pi))+\eta)^n$$
	for any $\eta>0$ and $n$ large enough.
	Consequently we have
	$$\lim_{n\to \infty} \left\|\dfrac{({L^\ep})^n (\varphi )}{\left(r_\sigma({L^\ep})\right)^n} -\Pi(\varphi) \right\|_{L^1({\Omega})}\leq  \lim_{n\to \infty}\left(\dfrac{r_\sigma({L^\ep}(I-\Pi))+\eta}{r_\sigma({L^\ep})}\right)^n\|\varphi\|_{L^1({\Omega})}=0$$
	where  $\eta>0$ is chosen such that $r_\sigma({L^\ep}(I-\Pi))+\eta<r_\sigma({L^\ep})$. This completes the proof of the first part of the lemma.
	\medskip

	{\bf{Step two:}}  	suppose first that $r_\sigma({L^\ep})<1$ and let $\varphi\in L^1_+({\Omega})$ be given. Due to the first item we have
	$$0=\limsup_{n\to \infty}\left\| \dfrac{{(L^\ep)}^n(\varphi)}{(r_\sigma({L^\ep}))^n}-\Pi(\varphi)\right\|_{L^1({\Omega})}\geq \limsup_{n\to \infty} \left\| \dfrac{{(L^\ep)}^n(\varphi)}{(r_\sigma({L^\ep}))^n}\right\|_{L^1({\Omega})}-\left\|\Pi(\varphi)\right\|_{L^1({\Omega})}.$$
	Assume by contradiction that
	$$\limsup_{n\to \infty}\|{(L^\ep)}^n(\varphi)\|_{L^1({\Omega})}>0.$$
	Then, there exist $\eta>0$ and a sequence $n_k\to\infty $  such that 
	\begin{equation*}
		\Vert (L^\ep)^{n_k}\Vert_{L^1(\Omega)}\geq \eta>0,\;\forall k\geq 0.
	\end{equation*}
	Therefore, we have 
	\begin{equation*}
		\frac{\eta}{(r_\sigma(L^\ep))^{n_k}}\leq \left\| \dfrac{{(L^\ep)}^{n_k}(\varphi)}{(r_\sigma({L^\ep}))^{n_k}}\right\|_{L^1({\Omega})}\leq \left\|\Pi(\varphi)\right\|_{L^1({\Omega})}+o(1)\text{ as $k\to\infty$}
	\end{equation*}
	which yields a contradiction.

	Consider now the case where $r_\sigma({L^\ep})>1$ and let $\varphi\in L^1_+({\Omega})$ be such that $\int_{{\Omega}}\varphi(y)\dd y>0$. Using again the part part of the lemma, we have 
	$$0=\limsup_{n\to \infty}\left\| \dfrac{{(L^\ep)}^n(\varphi)}{(r_\sigma({L^\ep}))^n}-\Pi(\varphi)\right\|_{L^1({\Omega})}\geq \left\|\Pi(\varphi)\right\|_{L^1({\Omega})}-\limsup_{n\to \infty} \left\| \dfrac{{(L^\ep)}^n(\varphi)}{(r_\sigma({L^\ep}))^n}\right\|_{L^1({\Omega})}.$$
	Assume by contradiction that
	$$\limsup_{n\to \infty}\|{(L^\ep)}^n(\varphi)\|_{L^1({\Omega})}<\infty.$$
	Then, there is $\eta>0$ and a sequence $n_k\to\infty $  such that 
	\begin{equation*}
		\Vert (L^\ep)^{n_k}\Vert_{L^1(\Omega)}\leq \eta<\infty,\;\forall k\geq 0.
	\end{equation*}
	Therefore, we have 
	\begin{equation*}
		\frac{\eta}{(r_\sigma(L^\ep))^{n_k}}\geq \left\| \dfrac{{(L^\ep)}^{n_k}(\varphi)}{(r_\sigma({L^\ep}))^{n_k}}\right\|_{L^1({\Omega})}\geq \left\|\Pi(\varphi)\right\|_{L^1({\Omega})}+o_{k\to\infty}(1),
	\end{equation*}
	which is a contradiction and item \ref{item:opnorm} is proved. This finishes the proof of Lemma \ref{Lemma:Prel_2}.
\end{proof}


\end{document}